\documentclass[11pt,a4paper,reqno]{amsart}
\usepackage{enumerate}
\usepackage{hyperref}
\usepackage{amsmath}
\usepackage[dvips]{graphicx}
\usepackage{color}

\usepackage{tikz-cd}  
\usepackage{setspace}  
\usepackage{amscd}
\usepackage{amsthm}
\usepackage{nicefrac}
\usepackage[all]{xy}
\usepackage{amssymb}
\newtheorem{theorem}{Theorem}[section]
\newtheorem{corollary}[theorem]{Corollary}
\newtheorem{lemma}[theorem]{Lemma}
\newtheorem{proposition}[theorem]{Proposition}

\theoremstyle{definition}
\newtheorem{definition}[theorem]{Definition}
\newtheorem{example}[theorem]{Example}
\theoremstyle{remark}
\newtheorem{remark}[theorem]{Remark}

\numberwithin{equation}{section}

\newcommand{\abs}[1]{\lvert #1 \rvert}

\let\epsilon\varepsilon
\renewcommand\emptyset\varnothing
\let\phi\varphi

\newcommand{\GL}{\mathrm{GL}}

\newcommand{\pr}{\mathrm{pr}}
\renewcommand{\Pr}{\mathrm{pr}}

\newcommand{\Sub}{\mathrm{Sub}}

\newcommand{\SL}{\mathrm{SL}}
\newcommand{\PSL}{\mathrm{PSL}}

\newcommand{\Sym}{\mathrm{Sym}}
\newcommand{\Ad}{\mathrm{Ad}}
\newcommand{\ad}{\mathrm{ad}}
\newcommand{\Lie}{\mathrm{Lie}}

\newcommand{\Aut}{\mathrm{Aut}}

\newcommand{\Sch}{\mathrm{Sch}}
\newcommand{\Cay}{\mathrm{Cay}}

\newcommand{\Boom}{\mathrm{Boom}}
\newcommand{\BoomlV}{\mathrm{Boom^{\lV}}}
\newcommand{\lV}{\mathrm{lV}}
\newcommand{\Out}{\mathrm{Out}}

\newcommand{\Stab}{\mathrm{Stab}}

\newcommand{\rk}{\mathrm{rk}}
\newcommand{\Comm}{\mathrm{Comm}}
\newcommand{\rd}{\mathrm{R}}

\newcommand{\Miss}{\operatorname{Miss}}
\newcommand{\Env}{\operatorname{Env}}
\newcommand{\diag}{\operatorname{diag}}

\newcommand{\fin}{\operatorname{f}}

\newcommand{\glofield}{\mathbb{K}}
\newcommand{\locfield}{\Bbbk}

\newcommand{\arrow}{\rightarrow}

\newcommand{\trivgp}{\langle e \rangle}
\newcommand{\defeq}{\stackrel{\mathrm{def}}{=}}
\newcommand{\R}{\mathbb R}

\newcommand{\Q}{\mathbb Q}
\newcommand{\N}{\mathbb N}
\newcommand{\Z}{\mathbb Z}

\newcommand{\G}{\mathbb G}
\renewcommand{\H}{\mathbb H}
\newcommand{\E}{\mathbb E}

\newcommand{\T}{\mathbb T}
\newcommand{\U}{\mathbb U}
\newcommand{\B}{\mathbb B}

\newcommand{\Lc}{\mathcal{L}}

\newcommand{\Bc}{\mathcal{B}}

\newcommand{\Oc}{\mathcal{O}}

\newcommand{\PP}{\mathbb P}

\newcommand{\frakg}{\mathfrak{g}}

\makeatletter
\newcommand\RSloop{\@ifnextchar\bgroup\RSloopa\RSloopb}
\makeatother
\newcommand\RSloopa[1]{\bgroup\RSloop#1\relax\egroup\RSloop}
\newcommand\RSloopb[1]%
  {\ifx\relax#1%
   \else
     \ifcsname RS:#1\endcsname
       \csname RS:#1\endcsname
     \else
       \GenericError{(RS)}{RS Error: operator #1 undefined}{}{}%
     \fi
   \expandafter\RSloop
   \fi
  }
\newcommand\X{0}
\newcommand\RS[1]%
  {\begin{tikzpicture}
     [every node/.style=
       {circle,draw,fill,minimum size=1.5pt,inner sep=0pt,outer sep=0pt},
      line cap=round
     ]
   \coordinate(\X) at (0,0);
   \RSloop{#1}\relax
   \end{tikzpicture}
  }
\newcommand\RSdef[1]{\expandafter\def\csname RS:#1\endcsname}
\newlength\RSu
\RSdef{i}{\draw (\X) -- +(90:\RSu) node{};}
\RSdef{l}{\draw (\X) -- +(135:\RSu) node{};}
\RSdef{r}{\draw (\X) -- +(45:\RSu) node{};}
\RSdef{I}{\draw (\X) -- +(90:\RSu) coordinate(\X I);\edef\X{\X I}}
\RSdef{L}{\draw (\X) -- +(135:\RSu) coordinate(\X L);\edef\X{\X L}}
\RSdef{R}{\draw (\X) -- +(45:\RSu) coordinate(\X R);\edef\X{\X R}}

\begin{document}

\title[Strong subgroup recurrence]{Strong subgroup recurrence and the Nevo--Stuck--Zimmer theorem}

\author{Yair Glasner} 
\address{Yair Glasner, Ben-Gurion University of the Negev,
	Departement of Mathematics,
	Be'er Sheva, 8410501, Israel.
	{\tt yairgl@bgu.ac.il}
}

\author{Waltraud Lederle} 
\address{Waltraud Lederle, TU Dresden, Institute of Algebra, Zellescher Weg 12-14, 01069 Dresden, Germany.
{\tt waltraud.lederle\@@tu-dresden.de}
}

\thanks{Both authors were partially funded by Israel Science Foundation grant ISF 2919/19.
The first author acknowledges funding from the F.R.S.-FNRS and UCLouvain for financing a visit to UCLouvain in January 2020.
The second author was partly supported by the Early Postdoc.Mobility scholarship No. 175106 from the Swiss National Science Foundation and was a F.R.S.-FNRS postdoctoral researcher.}


\date{\today}
\maketitle

\begin{abstract}
    Let $\Gamma$ be a countable group and $\Sub(\Gamma)$ its Chabauty space, namely the compact $\Gamma$-space consisting of all subgroups of $\Gamma$. We call a subgroup $\Delta \in \Sub(\Gamma)$ a boomerang subgroup if for every $\gamma \in \Gamma$, $\gamma^{n_i} \Delta \gamma^{-n_i} \rightarrow \Delta$ for some subsequence $\{n_i \} \subset \N$. Poincar\'{e} recurrence implies that $\mu$-almost every subgroup of $\Gamma$ is a boomerang, with respect to every invariant random subgroup $\mu$ of $\Gamma$. We establish for boomerang subgroups many density related properties, most of which are known to hold almost surely for invariant random subgroups. 
    
    Let $\glofield$ be a number field, $\Oc$ its ring of integers, $S$ a finite set of valuations including all the Archimedean valuations, and $\G$ an absolutely almost simple group defined over $\glofield$. Our main result is that if $\rk_\glofield \G \ge 2$ then any $\Gamma$ which is commensurable to the $S$-arithmetic group $\G(\Oc_S)$ has very few boomerang subgroups. Namely, every boomerang in $\Gamma$ is either finite and central or of finite index. In particular we recover Margulis' normal subgroup theorem as well as the Nevo-Stuck-Zimmer theorem for such lattices. 

    We include a short, accessible proof for the above theorem in the case that $\Gamma$ is commensurable to $\SL_n(\Z), \ n \ge 3$.
\end{abstract}
\tableofcontents

\section{Introduction}
Let $\Gamma$ be a countable discrete group and $\Sub(\Gamma)$ the space of all subgroups of $\Gamma$. The Chabauty topology, induced from the product topology on $\{0,1\}^{\Gamma}$, turns $\Sub(\Gamma)$ into a compact metrizable $\Gamma$-space, where the $\Gamma$-action is by conjugation. The fixed points for the dynamical system $(\Sub(\Gamma),\Gamma)$ are the normal subgroups of $\Gamma$, and the $\Gamma$-invariant Borel probability measures are known as {\it{invariant random subgroups}} or IRSs for short. If $\Gamma \curvearrowright (X,\Bc,\overline{\mu})$ is an action of $\Gamma$ on a standard probability space by measure preserving transformations (a p.m.p. action),  the stabilizer map $\Stab\colon X \rightarrow \Sub(\Gamma)$ defined by $\Stab(x) = \Gamma_x = \{\gamma \in \Gamma \ | \ \gamma x = x \}$ gives rise to an IRS $\mu = \Stab_{*}(\overline{\mu})$. It was shown in \cite{AGV:irs} that this construction is universal, in the sense that every IRS arises in this fashion. 

Consider the topological dynamical system $(\Sub(\Gamma),\Gamma)$. Restricting the action to any specific element $\gamma \in \Gamma$ yields a $\Z$-system. A point $\Delta \in \Sub(\Gamma)$ is called recurrent for this action if the forward orbit $\{\gamma^{n} \Delta \gamma^{-n} \mid n \in \N\}$ returns to every open neighborhood $\Delta \in U \subset \Sub(\Gamma)$ infinitely often. A subgroup $\Delta \in \Sub(\Gamma)$ is called a {\it{boomerang subgroup}} if it is a recurrent point with respect to every element of $\gamma \in \Gamma$. More explicitly a subgroup $\Delta \in \Sub(\Gamma)$ is a boomerang if for every $\gamma \in \Gamma$ and for every finite subset $S \subset \Gamma$ we have 
$$\Delta \cap S = \gamma^n \Delta \gamma^{-n}  \cap S \qquad  {\text{infinitely often.}}$$
If the weaker condition $\Delta \cap S \subset \gamma^n \Delta \gamma^{-n}  \cap S$ holds infinitely often for every $\gamma \in \Gamma$, we will refer to $\Delta$ as a lower Vietoris boomerang subgroup (lV-boomerang for short). Lower Vietoris boomerangs tend to admit useful stability properties that are not shared by the more restricted class of boomerang subgroups. We denote by $\Boom(\Gamma)$ and $\BoomlV(\Gamma) \subset \Sub(\Gamma)$ the subsets of boomerang and lV-boomerang subgroups, respectively.

Taking the intersection over all $\gamma \in \Gamma$, a Poincar\'{e} recurrence argument guarantees that $\mu(\Boom(\Gamma))=1$ for every IRS $\mu$ of $\Gamma$. Equivalently, for any p.m.p. action $\Gamma \curvearrowright (X,\Bc,\overline{\mu})$, we have $\Gamma_x \in \Boom(\Gamma)$ almost surely.

In fact we can go beyond p.m.p. actions.
We say that a nonsingular (i.e. measure class preserving) action $\Gamma \curvearrowright (X,\Bc,\overline{\mu})$ on a standard probability space is {\it{elementwise conservative}} (EC for short) if every $\gamma \in \Gamma$ acts as a conservative transformation; namely no $\gamma \in \Gamma$ admits a wandering set of positive measure. EC actions generalise the notion of p.m.p. actions. In an upcoming article with Tobias Hartnick we give many new examples of EC actions that do not come from p.m.p. actions, in the sense that there is no invariant measure in the given quasi-invariant measure class. The actions we construct are often totally non-free, hence give rise to boomerangs that do not come from IRSs, in the same sense. In a related upcoming paper, Le Ma\^{i}tre and Stalder use a similar construction to obtain free EC actions that do not come from p.m.p. actions.

Poincar\'{e} recurrence, which holds in the more general setting of conservative actions, again guarantees that $\Gamma_x \in \Boom(\Gamma)$ for $\overline{\mu}$-almost every $x \in X$. 
In particular, a statement that holds for every boomerang subgroup translates into a statement about all elementwise conservative actions, and in particular about all IRSs.
Our main theorem generalizes one of the central results about IRSs and p.m.p. actions to the more general setting of boomerangs and elementwise conservative actions:

\begin{theorem}  \label{thm:main}
Let $\G$ be an absolutely almost simple, connected, linear algebraic group defined over a number field $\glofield$. Let $S$ be a finite set of places of $\glofield$, containing all the infinite places, and $\mathcal{\Oc}_S$ the ring of $S$-integers in $\glofield$. Let $\Gamma$ be commensurable to the group $\G(\Oc_S)$. If $\rk_\glofield(\G) \ge 2$ then every lV-boomerang subgroup of $\Gamma$ either has finite index or is finite and central.
\end{theorem} 
As a corollary we obtain:
\begin{corollary} \label{cor:ec}
Let $\Gamma$ be as in Theorem \ref{thm:main}. Then every nonsingular ergodic elementwise conservative action $\Gamma \curvearrowright (X,\Bc,\overline{\mu})$ is either essentially free (up to a possible finite central kernel) or essentially supported on a finite orbit. 
\end{corollary}
This is a very interesting strengthening of the Nevo--Stuck--Zimmer theorem \cite{SZ:94}, which claims the same in the setting of p.m.p actions. One can see the generalization also in terms of the main theorem \ref{thm:main}: Margulis' normal subgroup theorem asserts that every normal subgroup $N \lhd \Gamma$ is either finite and central or of finite index. The Nevo--Stuck--Zimmer claims that the same holds, almost surely, for every IRS of $\Gamma$ and our theorem generalizes to boomerang subgroups.

Both of these famous deep theorems are known for a wider class of arithmetic groups than these treated in Theorem \ref{thm:main}. Most notably they do not require the higher $\Q$-rank assumption. Still, in the setting of Theorem \ref{thm:main} our new proof is surprisingly elementary and easy. In order to emphasize the simplicity of our argument we provide in Section \ref{sec:SLn} a short, self contained proof for the following corollary: 
\begin{corollary}\label{cor:sln}
Let $n \geq 3$ and $\Gamma$ be a finite index subgroup of $\SL_n(\Z)$. Then, every boomerang subgroup of $\Gamma$ is either contained in $\{\pm e\}$ or has finite index in $\Gamma$.
\end{corollary}

Another noteworthy corollary is the existence of {\it{boomerang simple}} groups. 
\begin{corollary} \label{cor:boom_simp}
In the setting of Theorem \ref{thm:main} assume that $\G$ is simply connected and that $\G(\glofield)$ is center free. Then $\Boom^{lV}(\G(\glofield)) = \{\G(\glofield),\trivgp\}$. This is the case for example for $\SL_n(\glofield)$ when $n \geq 3$ is odd and $\glofield \subset \R$. 
\end{corollary}
In particular a nonsingular, ergodic, nontrivial, elementwise conservative action of such a group is essentially free. This corollary will be deduced in Subsection \ref{sec:B_simple} from a slightly more general theorem stated there.

We  believe that the notion of boomerang subgroups is of independent interest. Boomerang subgroups can be interpreted as algebraic, deterministic counterparts to the more analytic notion of invariant random subgroups. It is interesting to ask which of the properties that are known to hold, almost surely, for IRSs are deterministically true for boomerangs. 
The following theorem can be thought of as a generalization of the Borel density theorem, or more precisely of the IRS version of this theorem \cite[Thm. 2.9]{7sam}.

\begin{theorem} \label{thm:BorelDensity}
Let $\G$ be a semi-simple, connected algebraic group defined over a field $\glofield$ of characteristic $0$ and $G = \G(\glofield)$. Let $\Gamma < G$ be a countable, Zariski dense subgroup. Assume that the kernel of the projection of $\Gamma$ to each simple factor of $G$ is finite. 
Then every lV-boomerang subgroup in $\Gamma$ is either finite and central or Zariski dense.
\end{theorem}
\begin{corollary}
In the setting of the above theorem, any nontrivial IRS\footnote{When we say that an IRS admits a certain property we assert that this property holds almost surely. In particular an IRS is nontrivial means that it is nontrivial almost surely.}  of a Zariski dense subgroup in $G$ is supported on Zariski dense subgroups.
\end{corollary}
Even though it is not explicitly stated there, this corollary could also be deduced from the main theorem in \cite{DGLL:cat0}. It is natural to ask for a geometric version of Borel density for boomerang subgroups, similar in nature to the one discussed in loc. cit. or in \cite{Osin:IRS_desity}. As an example in this direction we will prove the following.
\begin{theorem} \label{thm:convergence}
Let $\Gamma \curvearrowright X$ be a non-elementary faithful convergence action of a group $\Gamma$ on a compact space $X$ and $\trivgp \ne \Delta \in \Boom^{lV}(\Gamma)$, then $\Lc \Gamma = \Lc \Delta$. 
\end{theorem}

One of our main motivations for studying boomerangs was to provide a new angle of attack on unknown cases of the Nevo--Stuck--Zimmer theorem. Prior works \cite{SZ:94,NZ:IFT, NZ:proj_factors,Levit:NZ,Creutz:stabilizers} established this theorem for irreducible actions of higher rank semisimple Lie groups (also over some local fields) and their irreducible lattices, under the assumption that at least one of the simple factors has property $(T)$. The paper 
\cite{Bader_Shalom} treats general products $G = G_1 \times G_2$, where each $G_i$ is just-non-compact, not Euclidean with property $(T)$. The papers \cite{CP:stabilizers,BBHP:char_product}  treat groups of the form $\Gamma = \SL_2(\Z_S)$, where $S$ is any nonempty set of primes. Such groups are irreducible lattices inside the (restricted) product $G=\SL_2(\R) \times \prod^{r}_{p \in S}\SL_2(\Q_p)$, for which none of the factors has property $(T)$. But the question is still open for the group $G$ itself and other irreducible lattices there, as well as for products of other rank-$1$ groups without property $(T)$. Our hope is that the next theorem and the corollaries that follow would be a first step in this direction. 

\begin{theorem} \label{thm:rec_dense}
Let $\locfield$ be a non-Archimedean local field of characteristic zero (i.e. a finite extension of $\Q_p$). Let $\G(\locfield)$ be the $\locfield$-points of a $\locfield$-isotropic, simple, algebraic group over $\locfield$ and $G = \G(\locfield)^{+}$ the subgroup generated by all unipotent elements\footnote{This definition is good enough for us as we are working over fields of characteristic zero. Over non-perfect fields a slightly different definition for $G$ is required.}. Let $\Gamma < G$ be a countable dense subgroup of $G$. 
Then every non-central lV-boomerang $\Delta \in \Boom^{lV}(\Gamma)$ is also dense in $G$. \end{theorem}
In fact in Section \ref{sec:products} we state and prove more general  ``semisimple'' versions, Theorem \ref{thm:rec_dense_ss_p} and Corollary \ref{cor:rec_dense_ss}.
The IRS-version of this theorem does not seem to appear in the literature.
\begin{corollary}
    Under the assumptions of Theorem \ref{thm:rec_dense}, let $\mu$ be an ergodic IRS on $\Gamma$ that is not supported on central subgroups. Then $G = \overline{\Delta}$
%    \waltraud{This should be $=$ instead of $<$ right?}
    for $\mu$-almost every subgroup $\Delta < \Gamma$.
\end{corollary}

Theorem \ref{thm:rec_dense} has an immediate consequence for irreducible lattices.

 \begin{corollary} \label{cor:star}
 Let $\Gamma < \SL_2(\Q_p) \times \SL_2(\Q_q)$ be an irreducible lattice and $\Delta < \Gamma$ a non-central lV-boomerang subgroup.
 Then $\Delta$ has dense projections.
\end{corollary}

The IRS-version of this corollary is contained in the proof of \cite[Thm. 8.3]{Creutz:stabilizers}, but not stated there explicitly. Corollary \ref{cor:star} should be compared to \cite[Question 1.5]{FisherMjVanLimbeek2022}. There, the authors ask whether a discrete subgroup with dense projections in a product of semisimple groups over local fields already has to be a lattice. Recently a negative answer to their question was suggested in \cite{akhmedov2024irreduciblediscretesubgroupsproducts} in the form of a discrete infinite co-volume subgroup in $\SL_2(\R) \times \SL_2(\R)$ admitting dense projections. Still the question remains interesting for special families of groups. Particular interesting examples include infinite index subgroups in an irreducible lattice. Most interesting from our point of view would be any non-central infinite index boomerang subgroups  as in Corollary \ref{cor:star}. In this context, a positive result under some confinement assumptions was recently announced by Bader--Gelander--Levit.  

\subsection*{Comparison to other proofs.}
Our proof of Theorem \ref{thm:main} is algebraic in nature. It relies heavily on the $\glofield$-structure on the arithmetic group $\Gamma$ and  does not generalize to the case where $\rk_{\glofield}(\G) <2$. There were prior proofs, of algebraic nature, both for Margulis' normal subgroup theorem and to the Nevo--Stuck--Zimmer theorem. Indeed, even before Margulis, the normal subgroup theorem for $\Gamma = \SL_n(\Z), \ n\ge 3$ was established, by Bass--Lazzard--Serre \cite{BLS:CSP} and independently by Mennicke  \cite{Mennicke:CSP}, as corollary of their work on the congruence subgroup property. Bekka \cite{Bekka:CR} provided a beautiful algebraic proof for the Nevo--Stuck--Zimmer theorem as a corollary of his character rigidity theorem for the same group. 

%Our theorem, too, offers a generalization. Indeed, the family of  boomerang subgroups is in general richer than these coming from the support of IRSs. For example consider the group $\PSL_m(k)$, where $m \geq 2$ and $k$ is any locally finite countable field\footnote{Namely every finite set of elements in $k$ is contained in a finite subfield. An example is the algebraic closure of a finite field.}. Since the group is torsion every subgroup is a boomerang, while every IRS is supported on $\{\trivgp, \PSL_m(k)\}$, by \cite[Cor. 3.3]{PetersonThom2016}. See Remark \ref{rem:boom_non_irs} for more examples in this direction.

\subsection*{Organization of the paper.}
In Section \ref{sec:generalities} we give definitions and in Section \ref{sec:preliminaries} we discuss basic properties pertaining to boomerang subgroups. In Section \ref{sec:SLn} we present a short, elementary and self-contained proof of our main theorem for $\Gamma = \SL_n(\Z)$ for $n \ge 3$ and its finite index subgroups. In particular, this gives a very simple proof for the Nevo--Stuck--Zimmer theorem and for the Margulis normal subgroup theorem in this specific case. Theorem \ref{thm:BorelDensity} about Zariski density is proven in Section \ref{sec:BD}. The main theorem \ref{thm:main} and all of its corollaries are proven in Section \ref{sec:main}. The proofs here require some structure theory of simple Lie groups, but other than that they are not much more involved than the $\SL_n(\Z)$ case and follow the same strategy. Section \ref{sec:products} contains the proof of Theorem \ref{thm:rec_dense}.

\subsection*{Acknowledgements}
We thank Fran\c{c}ois Thilmany for his indispensable help with the intricacies of algebraic groups, for finding for us the beautiful but hard to find  work of Cutkosky and Srinivas \cite{CS1993}, and for numerous helpful discussions. We thank Adrien Le Boudec for suggesting Corollary \ref{cor:boudec} and Pierre-Emmanuel Caprace for helpful conversations and for hosting the visit of the first author at UCLouvain. Further thanks go to Tsachik Gelander for helpful comments and Gil Goffer for her careful reading and feedback.

Finally we thank two anonymous referees whose comments greatly improved the paper. The original version of the main theorem was restricted to the case where $\glofield = \Q$, $\G$ is $\Q$-split and $S = \{\infty\}$. Also the current version of Theorem \ref{thm:referee} and its proof are simpler and more general than our original proofs. We are thankful to the referees whose encouragement and many suggestions enabled the current improvements.  

\section{Boomerang subgroups} \label{sec:generalities}
Let $\Gamma$ be a countable group and $(X,\Gamma)$ a $\Gamma$-flow, i.e. $\Gamma$ acts on the compact topological space $X$ by homeomorphisms.
\begin{definition}
An element $\gamma \in \Gamma$ will be called a {\it{recurrent direction}} for a point $x \in X$ if the forward orbit $\{\gamma^n x \ | \ n \in \N\}$ returns to every open neighborhood $x \in O \subset X$. A point $x \in X$ is called a {\it{boomerang point}} if every $\gamma \in \Gamma$ is a recurrent direction for $x$. We denote by $\rd_{\Gamma}(x)$ the recurrent directions for a point $x$, and by $\Boom(X,\Gamma) = \{x \in X \ | \ \rd_{\Gamma}(x) = \Gamma\}$ the set of all boomerang points. 
\end{definition}
\begin{remark}
We will write $\rd_{\Gamma}^{\tau}(x)$ in order to make the dependence of the recurrent directions on the topology $\tau$ explicit. If $\tau_1 \subset \tau_2$ then $\rd_{\Gamma}^{\tau_2}(x) \subset \rd_{\Gamma}^{\tau_1}(x)$.  
\end{remark}
In other words, a point is a boomerang if it is a recurrent point, in the usual topological dynamical sense of the word, with respect to every element $\gamma \in \Gamma$ separately. Note that every torsion element of $\Gamma$ is automatically a recurrent direction for every $x \in X$. Clearly any $\Gamma$-equivariant map of such flows $\phi:X \rightarrow Y$ takes boomerang points to boomerang points.

%The following direct consequence of Poincar\'{e} recurrence relates the notion of boomerang points to EC actions as defined in the introduction: 
%\begin{theorem}
%Let $\Gamma$ be a countable group $(X,\Gamma)$ a $\Gamma$ flow on a second countable compact space $X$. If $\mu \in \operatorname{Prob}(X)$ is a quasi-invariant measure such that every $\gamma \in \Gamma$ acts on $(X,\mu)$ as a conservative transformation. Then $\mu$-almost every $x \in X$ is boomerang.
%\end{theorem}
%Note that the EC condition holds automatically whenever $\mu$ is a $\Gamma$-invariant probability measure. 
%\begin{proof}
%By Poincar\'{e} recurrence, for every fixed subset $O \subset X$ and $\gamma \in \Gamma$ we have 
%$$\mu\left(X \setminus O \cup \left\{x \in O \ | \ \{\gamma^n x \ | \ n \in \N\} \cap O \ne \emptyset \right\} \right) = 1.$$
%Now take the intersection over all $\gamma \in \Gamma$ and over all $O$ ranging over a countable basis for the topology on $X$. 
%\end{proof}

Now let $\Sub(\Gamma)$ be the space of all subgroups of $\Gamma$. There are two relevant compact topologies on this set: 

\begin{definition} For a countable group $\Gamma$ and a subset $F \subset \Gamma$, we denote
\begin{eqnarray*}
\Env(F) & = & \{\Delta \in \Sub(\Gamma) \ | \ F \subset \Delta\} \\
\Miss(F) & = & \{\Delta \in \Sub(\Gamma) \ | \ \Delta \cap F = \emptyset\}.
\end{eqnarray*}
The collection of sets $\{\Env(E) \cap \Miss(F) \ | \ |E|,|F|<\infty\}$ forms a basis for the \emph{Chabauty topology} on $\Sub(\Gamma)$. The collection $\{\Env(F) \ | \ \abs{F} < \infty\}$ forms a basis for the \emph{lower Vietoris topology}. 
\end{definition}

\begin{remark} The upper Vietoris topology is defined similarly using the basis $\{\Miss(F) \ | \ \abs{F}<\infty\}$. One interesting feature is that for any topological dynamical system  $(X,\Gamma)$, the stabilizer map 
$$X \rightarrow \Sub(\Gamma), \qquad x \mapsto \Gamma_x$$
is continuous when $\Sub(\Gamma)$ is endowed with the upper Vietoris topology. We mention this for the sake of completeness, but for us, the upper Vietoris topology will not be useful and we will not appeal to it again in this paper. 
\end{remark}
The Chabauty topology coincides with the Tychonoff topology induced on the closed subset $\Sub(\Gamma) \subset \{0,1\}^{\Gamma}$. In particular it is compact and metrizable. The lower Vietoris topology, which by definition is coarser, is still compact. But it is far from being Hausdorff or even $T_1$. In fact $\overline{\{\Delta\}}^{lV} = \Sub(\Delta) \subset \Sub(\Gamma)$, so the only closed point is the trivial subgroup.

\begin{remark}
For a general subgroup $\Delta < \Gamma$ the set $\rd_\Gamma(\Delta)$ is not a subgroup of $\Gamma$.
For example, if $\Gamma = \langle a,b \rangle$ is the free group on two generators and $\Delta = \langle a^2, b^2 \rangle < \Gamma$ then clearly $a, b \in \rd_{\Gamma}(\Delta)$ but $ab \notin \rd_\Gamma(\Delta)$.
\end{remark}

Now we give the main definition of this article.

\begin{definition}
  Let $\Gamma$ be a countable discrete group. A subgroup $\Delta \in \Sub(\Gamma)$ is called a {\it{boomerang subgroup}}, or simply a \emph{boomerang}, if it is a boomerang point of $\Sub(\Gamma)$ with respect to the Chabauty topology. Similarly, it is called an \emph{lV-boomerang subgroup}, or an \emph{lV-boomerang} for short, if it is a boomerang point of $\Sub(\Gamma)$ with respect to the lower Vietoris topology. We denote the corresponding subsets of $\Sub(\Gamma)$ by $\Boom(\Gamma)$ and $\BoomlV(\Gamma)$. 
\end{definition}

\begin{remark}
Lower Vietoris boomerangs already appeared, and played an important role, in the paper \cite[Def. 2.4]{Gla:lin_irs}. There they were called {\it{recurrent subgroups}}. We decided to change their name in order to avoid conflict with other uses of recurrence in topological dynamics. 

There is a tension between the concepts of boomerang and lV-boomerang subgroups. The word boomerang was paired with the Chabauty topology, which is the canonical topology on $\Sub(\Gamma)$, and is connected to IRSs via Poincar\'e recurrence. In this way boomerangs seem to be more natural. On the other hand lV-boomerangs are the more general objects and it is important that we can prove all of our theorems for them. Aside from their wider generality, Lower Vietoris boomerang subgroups exhibit better stability properties. Lemma \ref{lem:surj_image} below shows that the image of an lV-boomerang subgroup $\Delta < \Gamma$, under a surjective homomorphism $\varphi:\Gamma \rightarrow \Sigma$, is again an lV-boomerang. Example \ref{eg:boom_image} shows that $\varphi(\Delta)$ need not be boomerang; even when $\Delta \in \Boom(\Gamma)$. Applying such a surjective homomorphism we need to pass to a lV-boomerang subgroup in the proof of Theorem \ref{thm:rec_dense} and it becomes crucial that all theorems so far have been proved in this generality. Theorem \ref{thm:rec_dense} in turn is used in the proof of the Main theorem \ref{thm:main}. Thus their stronger robustness property becomes an important reason to keep track of lV-boomerangs.
\end{remark}

Let us recall the following definition from the introduction
\begin{definition} \label{def:ec}
A measurable action of a group on a standard probability space $\Gamma \curvearrowright (X,\Bc,\mu)$ is called {\it{elementwise conservative}} if it is nonsingular (i.e. measure class preserving) and  every $\gamma \in \Gamma$ acts as a conservative transformation. 
\end{definition}
\begin{example} \label{eg:boom}
Here are a few sources for boomerang subgroups: 
\begin{enumerate}
\item \label{itm:n} normal subgroups.
\item \label{itm:nfi} finite index subgroups, or more generally subgroups whose normalizer is of finite index. 
\item \label{itm:lf+} Every subgroup of a torsion group  is boomerang.
\item \label{itm:boring} subgroups whose normalizer contains a positive power of every element; these are exactly the subgroups of $\Gamma$ whose orbit under the $\gamma$-action is finite for every $\gamma \in \Gamma$.
\item \label{itm:irs} If $\mu$ is an IRS of $\Gamma$ then $\mu$-almost every subgroup of $\Gamma$ is a boomerang. 
\item \label{itm:ec} Let $\Gamma \curvearrowright (X,\Bc,\overline{\mu})$ be an elementwise conservative action. Then the stabilizer $\Gamma_x \in \Boom(\Gamma)$ for $\overline{\mu}$-almost every $x \in X$.
\end{enumerate}
\end{example}
\begin{proof}
(\ref{itm:boring}) is clear, and (\ref{itm:n}), (\ref{itm:nfi}) and (\ref{itm:lf+}) are special cases of it. Then, (\ref{itm:irs}) is a special case of (\ref{itm:ec}), so it remains to show (\ref{itm:ec}).

The map $\Stab\colon X \rightarrow \Sub(\Gamma)$ is Borel and $\Gamma$-equivariant. We set $\mu = \Stab_*(\overline{\mu})$.
Thus $(\Sub(\Gamma),\mu)$ is an elementwise conservative $\Gamma$-space. 
% Fix $\gamma \in \Gamma$ and let $\Sub(\Gamma) = C(\gamma) \sqcup D(\gamma)$ be the corresponding Hopf decomposition. Since $\Stab^{-1}(D(\gamma))$ is $\gamma$-dissipative we have $\mu(D(\gamma)) = \overline{\mu}(\Stab^{-1}(D(\gamma))=0$. 
Now, for every $\gamma \in \Gamma$,  Poincar\'{e} recurrence guarantees that $\mu$-almost every $\Delta \in \Sub(\Gamma)$ is a $\gamma$-recurrent point.
Since $\Gamma$ is countable, $\mu$-almost every subgroup of $\Gamma$ is a boomerang subgroup.
Hence $1 = \mu(\Boom(\Gamma)) = \overline{\mu}(\Stab^{-1}(\Boom(\Gamma))) = \overline{\mu}(\{x \in X \mid \Gamma_x \in \Boom(\Gamma)\})$.
%
% Taking intersection over all $\gamma \in \Gamma$ and pulling back to $X$ we deduce that $\Gamma_x \in \Boom(\Gamma)$ for $\overline{\mu}$-almost every $x \in X$ \waltraud{rewrite}. 
%
%It is easy to verify (\ref{itm:nfi}) directly, but it also follows from (\ref{itm:irs}).  Poincar\'{e} recurrence implies that for every $\gamma \in \Gamma$
%the set $\{\Delta \in \Sub(\Gamma) \mid \gamma \in \rd_\Gamma(\Delta) \}$ has full measure.
%But $\Gamma$ is countable, and therefore $\bigcap_{\gamma \in \Gamma} \{\Delta \in \Sub(\Gamma) \mid \gamma \in \rd_\Gamma(\Delta) \}$ has full measure.
%This completes the proof of (\ref{itm:irs}). The proof of (\ref{itm:lf+}) is clear. {\color{blue} For the proof of (\ref{itm:ec}) see Corollary \ref{cor:ec} in the introduction.}
\end{proof}
In view of this, Corollary \ref{cor:ec} follows directly from our main theorem:
\begin{proof}[Proof of Corollary \ref{cor:ec}]
Let $\Gamma \curvearrowright (X,\Bc,\overline{\mu})$ be an EC action. From Example \ref{eg:boom} (\ref{itm:ec}) $\Gamma_x \in \Boom(\Gamma)$ for $\overline{\mu}$-almost every $x \in X$. In view of Theorem \ref{thm:main} we deduce that for $\mu$-almost every $x \in X$ the group $\Gamma_x$ is either of finite index or finite and central. In the first case, ergodicity implies that there is one fixed finite central kernel for the action. In the second case, ergodicity implies that the measure is supported on a finite orbit. 
\end{proof}

\begin{remark}
As mentioned in the introduction, two upcoming papers, one joint with Tobias Hartnick and another one of Le Ma\^{i}tre and Stalder, give examples of EC actions. Many of these give rise to boomerangs that do not come from p.m.p. actions.
\end{remark}
\begin{remark} 
\label{rem:boom_non_irs}
One ad-hoc construction of boomerang subgroups that do not come from p.m.p. actions utilizes Example (\ref{itm:lf+}) above. In fact the abundance of boomerangs in a torsion group $\Gamma$ can be pulled back via an epimorphism $\phi \colon F_{\infty} \twoheadrightarrow \Gamma$ to the free group $F_{\infty}$. Here, every subgroup inside the closed subset $$\Env(\ker(\phi)) = \{\Delta \in \Sub(F_{\infty}) \ | \ \ker(\phi) < \Delta\}$$ is boomerang.
 In contrast, by \cite[Cor. 3.3]{PetersonThom2016}, if $\Gamma = \PSL_m(k)$, where $m \geq 2$ and $k$ is a locally finite countable field, every IRS supported on this set is of the form $\lambda \delta_{\ker(\phi)} + (1-\lambda)\delta_{F_{\infty}}, \ \lambda \in [0,1]$. 
 
 Similarly, if $\Gamma$ is a Tarski monster, it cannot have any IRSs because it has only countably many subgroups, so every ergodic IRS is supported on a finite orbit, and there are no such in $\Sub(\Gamma)$ other than the two trivial ones. Thus we can replace the free group above by a finitely generated free group. In fact, since any non-elementary Gromov hyperbolic group $\Gamma$ maps onto a Tarski monster $\phi \colon \Gamma \twoheadrightarrow T$, we can find in any such group a closed set $\Env(\ker(\phi)) \subset \Sub(\Gamma)$ with the same properties as above. Note that $\Env(\ker(\phi))$, which is naturally homeomorphic to $\Sub(T)$, is countable in this case.  
\end{remark}

\begin{remark} 
The notion of boomerang subgroups can also be defined for a second countable, locally compact group $G$. A subgroup $\Delta < G$ in a topological group is called a {\it{boomerang subgroup}} if $\rd_{G}(\Delta)$ is dense in $G$. With this definition, all the Examples \ref{eg:boom} remain valid, for the same reasons. Note that to retain the basic property that every IRS is supported on $\Boom(G)$, this topological version of the definition is needed; even in the most basic examples. For example $\diag(a,a^{-1})$ with $a > 1$ fails to be a recurrent direction for $\SL_2(\Z) \in \Sub(\SL_2(\R))$. On the other hand, it turns out that $\rd_{G}(\Delta)$ is always a $G_{\delta}$-subset of $G$, so that when $\Delta$ is a boomerang $\rd_{G}(\Delta)$ is not merely dense, it is a residual set. We plan to work on boomerangs in locally compact groups in the future, and mention this here only in passing.
\end{remark}

\subsection*{Visualization via Schreier graphs} %\label{eg:schreier} 
Let $\Gamma = \langle S \rangle$ be a finitely generated group, with a given choice of a finite symmetric generating set $S$. One can identify, in a natural way, $\Sub(\Gamma)$ with the space of all pointed connected $S$-Schreier graphs of $\Gamma$. For any $\gamma \in \Gamma$, written as a word in the generators $\gamma = \omega(S)$, the conjugation action $\Delta \mapsto \gamma \Delta \gamma^{-1}$ on $\Sub(\Gamma)$ is then identified with the basepoint shift action $(X,o) \mapsto (X,\gamma o)$ at the level of Schreier graphs. Graphically, one passes from the vertex $o$ to $\gamma o$ by following the labeled edges according to the word $\omega(S)$.

With this terminology in place, a pointed Schreier graph $(X,o)$ corresponds to a boomerang if and only if, for every $R>0$ and every $\gamma = \omega(S)$, following the sequence of vertices 
$$o_0=0, \,\, o_1=\omega(S) o, \,\, o_2 =\omega^2(S)o, \,\,\ldots, \,\, o_n =\omega^n(S)o, \,\,\ldots$$
we are going to find some $n$ such that the balls 
$B_{X}(o,R) \cong B_{X}(o_n,R)$ are isomorphic. This isomorphism is in the category of Schreier graphs, namely it respects the base point and the edge labels. 

Similarly, a Schreier graph $(X,o)$ as above represents an lV-boomerang if and only if for every $R >0$ we can find an $n$ with a center preserving, edge label preserving graph homomorphism $\psi \colon B_{X}(o,R) \to B_{X}(o_n,R)$.

\begin{example} \label{eg:lV_boom}
Let $F_2 = \langle a,b\rangle$ be the free group with free generators $a,b$.
Define $\varphi \colon F_2 \rightarrow \Z$ via $\varphi(a)=1, \varphi(b)=0$, and take $N = \ker(\varphi)$. Since $N$ is normal it is clearly a boomerang, and it is easy to check that $N$ is freely generated by $\{n_i = a^nba^{-n} \ | \ n \in \Z\}$. 
Now let us define two homomorphisms $\eta_i \colon N \rightarrow \Z/2\Z$ via 
$$\eta_1(n_i) = \left\{ \begin{matrix} 1 & n = 0 \\ 0 & n \ne 0 \end{matrix} \right. , \qquad \qquad 
\eta_2(n_i) = \left\{ \begin{matrix} 0  & n = 0 \\ 1 & n \ne 0 \end{matrix} \right. .$$
Set also $\Delta_i = \ker(\eta_i)$, these are normal index two subgroups of $N$. 
The Schreier graph of $\Delta_1$ is depicted in Figure \ref{fig:Schreier}. The Schreier graph of $\Delta_2$ is obtained from that of $\Delta_1$ by flipping the roles of the two types of $b$ edges.
It is easy to verify, using the above criteria, that $\Delta_1$ is an lV-boomerang but not a boomerang subgroup of $F_2$. On the other hand, $\Delta_2$ is not even an lV-boomerang.
\begin{figure}
\begin{tikzcd}
    \cdots\ar[r,"a"]& 
    \bullet\ar[r,"a"]\ar[loop,"b",->, in=60,out=120,looseness=12]& 
    \bullet\ar[r,"a"]\ar[loop,"b",->, in=60,out=120,looseness=12]& 
    \bullet\ar[r,"a"]\ar[d,"b" out=-120,bend left]& 
    \bullet\ar[r,"a"]\ar[loop,"b",->, in=60,out=120,looseness=12]& 
    \bullet\ar[r,"a"]\ar[loop,"b",->, in=60,out=120,looseness=12]& 
    \cdots 
    \\
    \cdots\ar[r,"a"]& 
    \bullet\ar[r,"a"]\ar[loop,"b",->, in=-60,out=-120,looseness=12]& 
    \bullet\ar[r,"a"]\ar[loop,"b",->, in=-60,out=-120,looseness=12]& 
    \star \ar[r,"a"]\ar[u,"b" out=-120,bend left]& 
    \bullet\ar[r,"a"]\ar[loop,"b",->, in=-60,out=-120,looseness=12]& 
    \bullet\ar[r,"a"]\ar[loop,"b",->, in=-60,out=-120,looseness=12]& 
    \cdots 
\end{tikzcd}
\caption{The Schreier graph of an lV-boomerang which is not a boomerang.}
\label{fig:Schreier}
\end{figure}
\end{example}

As promised, we now use Schreier graph methods to demonstrate that the image of a boomerang subgroup under a surjective homomorphism $\varphi 
\colon \Gamma \rightarrow \Sigma$ need not be a boomerang anymore.

\begin{example}\label{eg:boom_image}
Let $\Sigma = \langle a,b \rangle$ and $\Gamma = \langle a,b,c \rangle$ be non-abelian free groups. Define a projection $\varphi\colon\Gamma 
 \rightarrow \Sigma$ by setting $\varphi(c)=1$. Then there is $\Delta \in \Boom(\Gamma)$ such that $\varphi(\Delta) = \Delta_1$ is the lV-boomerang subgroup that is not boomerang given in Example \ref{eg:lV_boom}.

Set $\Delta_2 := \phi^{-1}(\Delta_1)$.
The Schreier graph $\Sch(\Gamma,\Delta_2,S)$, with $S=\{a,b,c\}$, is depicted at the bottom of Figure \ref{fig:boom_image} with base point $\star$.

We proceed to define the group $\Delta$ via its Schreier graph $X = \Sch(\Gamma, \Delta, S)$,  restricting ourselves to graphs satisfying the following properties: 
\begin{itemize}
\item Erasing all the $b$-labeled edges, will leave a disjoint union of trees, all isomorphic to the Cayley tree of $\langle a,c \rangle$.
\item The $b$-labeled edges will always assume either one of two possible forms:

\begin{center}
a {\it{dagger}} 
\begin{tikzcd}
    \bullet \ar[r,"b" out=-120,bend left] & 
    \bullet \ar[l,"b" out=-120,bend left]
\end{tikzcd}
or a {\it{loop}} 
\begin{tikzcd} 
    \bullet \ar[loop,"b",->, in=60,out=120,looseness=4]
\end{tikzcd}
\end{center}
\item There exists morphism $f \colon \Sch(\Gamma,\Delta,S) \to \Sch(\Gamma,\Delta_2,S)$ of Schreier graphs.
% \begin{center}
% \begin{tikzcd}
% & \Cay(\Gamma,S) \arrow [dl,"\pi"']\arrow [dr,"\pi_1"]  &  \\
% \Sch(\Gamma,\Delta,S) \arrow [rr,"f"] & & \Sch(\Gamma,\Delta_2,S) 
% \end{tikzcd}
% \end{center}
\end{itemize}
It is known from covering space theory that the existence of $f$ as above is equivalent 
to $\Delta < \Delta_2$, i.e. equivalent to $\phi(\Delta) < \Delta_1$.

Assume we constructed $X$ satisfying these conditions, we now explain what we need in order to have $\phi(\Delta)=\Delta_1$.
Consider the obvious inclusion $i \colon \Sch(\Sigma,\Delta_1,\{a,b\}) \to \Sch(\Gamma,\Delta_2,S)$. Then, $\phi(\Delta) = \Delta_1$ if and only if for each closed path $\alpha \in \pi_1(\Sch(\Sigma,\Delta_1,\{a,b\}),\star)$ there exists a closed path $\tilde \alpha \in \pi_1(\Sch(\Gamma,\Delta,S),\diamond)$ with $i^{-1}(f(\tilde \alpha)) = \alpha$.
%\waltraud{Say here immediately what condition is needed for $\phi(\Delta)=\Delta_1$.}

After having established what we need to achieve, we proceed to define such graphs, and the map $f$, inductively.
We are actually doing a probabilistic construction, meaning that some choices are random, but almost surely will deliver what we need.
We start with the obvious map $f_1\colon \Cay(F_2,\{a,c\}) \rightarrow \Sch(\Gamma,\Delta_2,S)$. Now assume we have already constructed 
$f_n\colon X_n \rightarrow \Sch(\Gamma,\Delta_2,S)$ with $X_n$ a partially defined Schreier graph satisfying the conditions above but with some of the vertices not incident to $b$-labeled incoming and outgoing edges. Let's refer to such vertices as vacant vertices. To define $f_{n+1}\colon X_{n+1} \rightarrow \Sch(\Gamma,\Delta_2,S)$ we choose, for each vacant vertex $x$ separately, between two possibilities:
\begin{itemize}
\item[L.] Attach a loop.
\item[D.] Attach a dagger, with a new $\Cay(F_2,\{a,c\})$ attached to its other side, and extend the map $f_{n+1}$ to the new tree in the only possible way.
\end{itemize}
When making this choice, we only have to obey the following rule:
If $f_n(x)$ is one of the two vertices incident to the dagger in $\Sch(\Gamma,\Delta_1,S)$ we must choose a dagger; this will ensure that the map $f_{n+1}$ is well defined. In all other cases we decide by tossing a $\{D,L\}$-coin, with all coins along the process chosen uniformly and independently. 

Figure \ref{fig:boom_image} depicts the first step of this process. The basepoint is depicted by $\diamond$. To emphasize the difference between the two types of vacant vertices: daggers attached to new trees are drawn only at vertices where they must appear. The new tree at the other end of every such dagger is denoted symbolically by $\RS{lrIlir}$. On all other vacant vertices it should be decided randomly whether to attach a $b$-loop or a $b$-dagger with a whole new tree at its other end.

It is clear that $f_n$ is surjective for every $n \ge 2$. It is also easy to verify that $X = \bigcup_{n\in\N}X_{n}$ is indeed a Schreier graph of a group  $\Delta \in \Sub(\Gamma)$ and the map $f = \bigcup_{n \in \N} f_n \colon X \rightarrow \Sch(\Gamma,\Delta_2,S)$ guarantees that $\varphi(\Delta) \subset \Delta_1$. Both $X=X_{\omega}$ and $\Delta= \Delta_{\omega}$ are random objects, depending on a point $\omega \in \Omega$ in a probability space signifying all of our choices, but for better readability we omit $\omega$ from the notation. 

To verify that $\varphi(\Delta) = \Delta_1$ we have to check that
each closed path starting at $\star$ in the Schreier graph of $\Delta_1 < \Sigma$ can be enriched with $c$-loops so that the resulting closed path, which is now a closed path in $\Sch(\Gamma,\Delta_2,S)$, lifts to a closed path in $\Sch(\Sigma,\Delta,S)$.
But this is possible as soon as
$f^{-1}(x)$ contains at least one vertex incident to a loop for every vertex $x \in \Sch(\Gamma,\Delta_2,S)$ that is not attached to the dagger. We claim that this holds almost surely. Since $f$ is surjective $f^{-1}(x)$ is not empty. Hence $f^{-1}(x)$ is infinite, as every vertex there is involved in a $\{a,c\}$-tree. The probability that such an infinite locus involves only daggers is zero.

We claim that $\Delta$ is almost surely a boomerang. For this fix a radius $R$ and an element $\gamma = w(a,b,c)$. Denoting by $\diamond = e\Delta \in \Sch(\Gamma,\Delta,S)$, $\star = e\Delta_1 \in \Sch(\Gamma,\Delta_1,S)$ the basepoints of our two Schreier graphs, we have to prove that for some $n$ the two $R$-balls $B_X(\diamond,R)$ and  $B_X(\gamma^n \diamond,R)$ are isomorphic as edge labeled graphs.  Let $\zeta\colon \Gamma \rightarrow \Z$ be given by $\zeta(a)=1,\zeta(b)=\zeta(c)=0$. Note that $\zeta(\gamma')=0$ if and only if $\gamma' \star$ is one of the two vertices incident to the dagger, and for $n \geq 0$
$$d(\star, f(\gamma^n \diamond)) \in \left \{n |\zeta(\gamma)|, n |\zeta(\gamma)| +1 \right \}.$$

\noindent {\it{Case 1:}} $\gamma^{n} \diamond = \diamond$ for some $n$. In this case there is nothing to prove. 

This case is equivalent to the $\langle \gamma \rangle$-orbit of $\diamond$ being bounded. Thus in the two cases that follow we can find a sequence of points along the orbit $\gamma^{n_i} \diamond$ such that the $R$-balls around them are disjoint and thus independently distributed. They are not identically distributed though, which is exactly the reason behind the separation into two cases.

\noindent {\it{Case 2:}} $\zeta(\gamma) = 0$ but the $\langle \gamma \rangle$-orbit of $\diamond$ is unbounded. In this case $f(\gamma^n \diamond)$ is always one of the two vertices incident to the dagger in $\Sch(\Gamma,\Delta_1,S)$; a fact that imposes a lot of restrictions on the possible shape of the balls $B_X(\gamma^{n_i} \diamond,R)$. Still, the specific shape of $B_X(\diamond,R)$ is definitely a legal one and thus bound to re-occur infinitely often.

\noindent {\it{Case 3:}} $\zeta(\gamma)\ne 0$. Here the balls 
$B_X(\gamma^{n_i} \diamond,R)$ can be chosen in such a way that their images under $f$ are all disjoint. Thus, for $i>0$, $f(B_X(\gamma^{n_i} \diamond,R))$ never touches the dagger, and the distribution of such balls is much more free. In particular, the specific shape of $B_X( \diamond,R)$ is still legal and hence bound to reoccur infinitely often. 
This completes the proof.

\begin{figure}[ht] \label{fig:boom_image}
\begin{tikzcd} [sep = .5 cm]
& & & & & & & \RS{lrIlir} \arrow [dl,"b"',bend left]  & & & & &  \\
& & & & & \bullet \arrow [r,"a"] & \bullet \arrow [r,"a"'] \arrow [ur,bend left]  & \bullet & & & & & \\
& & & & & & & & & & & & \\
& & & \bullet & \RS{lrIlir} \arrow [d,"b"',bend left] & & & \RS{lrIlir} \arrow [dl,"b"',bend left] & \RS{lrIlir} \arrow [d,"b"',bend left] & \bullet & & & \\
& &  \bullet \arrow [r,"a"] & \bullet \arrow [r,"a"]  \arrow [u,"c"]  & \bullet \arrow [u,bend left] & \bullet \arrow [r,"a"] & \bullet \arrow [r,"a"']  \arrow [uuu,"c"] \arrow [ur,bend left]  & \bullet & \bullet \arrow [u,bend left] \arrow [r,"a"] & \bullet \arrow [r,"a"] \arrow [u,"c"] & \bullet & & \\
& & & & & & & & & & & & \\
& & & & & & & \RS{lrIlir} \arrow [dl,"b"',bend left] & & & & & \\
\bullet \arrow [rrr,"a"] & & & \bullet \arrow [rrr,"a"]  \arrow [uuu,"c"] & & & \diamond \arrow [rrr,"a"] \arrow [uuu,"c"] \arrow [ur,bend left] & & & \bullet \arrow [rrr,"a"] \arrow [uuu,"c"] & & & \bullet \\
& & & & & & & & & & & & \\
& & & & \RS{lrIlir} \arrow [d,"b"',bend left] & & & \RS{lrIlir} \arrow [dl,"b"',bend left] & \RS{lrIlir} \arrow [d,"b"',bend left] & & & & \\
& &  \bullet \arrow [r,"a"] & \bullet \arrow [r,"a"]  \arrow [uuu,"c"]  & \bullet \arrow [u,bend left] & \bullet \arrow [r,"a"] & \bullet \arrow [r,"a"']  \arrow [uuu,"c"] \arrow [ur,bend left]  & \bullet & \bullet \arrow [u,bend left] \arrow [r,"a"]  & \bullet \arrow [r,"a"] \arrow[uuu,"c"] & \bullet & & \\
& & & \bullet \arrow[u,"c"] & & &  \bullet \arrow [u,"c"] \arrow[r,"b"', bend left]  & \RS{lrIlir} \arrow[l,bend left] & & \bullet \arrow[u,"c"]  & & & \\
& & & & & & \phantom{a} \arrow [ddd,"f"] & & & & & & \\
& & & & & &  & & & & & & \\
%& & & & & &  & & & & & & \\
& & & & & & & & & & & & \\
& & & & & & \phantom{A} & & & & & & \\
    \cdots\ar[rr,"a"]& & 
    \bullet\ar[rr,"a"]\ar[loop,"b",->, in=30,out=150,looseness=12] \ar[loop,"c",->, in=60,out=120,looseness=6]& & 
    \bullet\ar[rr,"a"]\ar[loop,"b",->, in=30,out=150,looseness=12] \ar[loop,"c",->, in=60,out=120,looseness=6]& & 
    \bullet\ar[rr,"a"]\ar[d,"b" out=-120,bend left] \ar[loop,"c",->, in=60,out=120,looseness=6]& &
    \bullet\ar[rr,"a"]\ar[loop,"b",->, in=30,out=150,looseness=12] \ar[loop,"c",->, in=60,out=120,looseness=6]& &
    \bullet\ar[rr,"a"]\ar[loop,"b",->, in=30,out=150,looseness=12] \ar[loop,"c",->, in=60,out=120,looseness=6]& &
    \cdots 
    \\
    \cdots\ar[rr,"a"]& &
    \bullet\ar[rr,"a"]\ar[loop,"b",->, in=-30,out=-150,looseness=12] \ar[loop,"c",->, in=-60,out=-120,looseness=6]& & 
    \bullet\ar[rr,"a"]\ar[loop,"b",->, in=-30,out=-150,looseness=12] \ar[loop,"c",->, in=-60,out=-120,looseness=6]& &
    \star \ar[rr,"a"]\ar[u,"b" out=-120,bend left] \ar[loop,"c",->, in=-60,out=-120,looseness=6]& &
    \bullet\ar[rr,"a"]\ar[loop,"b",->, in=-30,out=-150,looseness=12] \ar[loop,"c",->, in=-60,out=-120,looseness=6]& &
    \bullet\ar[rr,"a"]\ar[loop,"b",->, in=-30,out=-150,looseness=12] \ar[loop,"c",->, in=-60,out=-120,looseness=6]& & 
    \cdots 
\end{tikzcd}
\caption{Parts of the Schreier graphs of $\Delta$ and $\Delta_2$}
\end{figure}
\end{example}

\subsection*{Basic properties of boomerang subgroups} \label{sec:preliminaries}
In this section we establish some basic properties of boomerang subgroups. We will often use the following without explicitly referring to it.

\begin{lemma}\label{lem:commutator_in_boom}
    Let $\Delta \in \BoomlV(\Gamma)$, $\delta \in \Delta$, $\gamma \in \Gamma$ and $n \in \N$. Then, for infinitely many $k>0$ we have $\gamma^k \delta \gamma^{-k} \in \Delta$ and $[\gamma^k,\delta^n] \in \Delta$.
\end{lemma}

\begin{proof}
Since $\gamma^{-1}$ is an lV-recurrent direction, 
there are infinitely many $k > 0$ such that
$\gamma^{-k} \Delta \gamma^k \in \Env(\delta)$,
i.e.  there is a $\delta_k \in \Delta$ with $\gamma^{-k} \delta_k \gamma^{k} = \delta$.
But then $\gamma^k \delta \gamma^{-k} = \delta_k \in \Delta$
%and consequently $[\gamma^k,\delta^{-1}]=\delta \delta_k \in \Delta$
and consequently
$[\gamma^k,\delta^n] = \delta_k^n \delta^{-n} \in \Delta$.
\end{proof}

\begin{lemma}\label{lem:rec_intersection_subgroup}
Let $\Gamma_2$ be a countable discrete group and $\Gamma_1 < \Gamma_2$ be a subgroup. 
Let $\Delta < \Gamma_2$ be a subgroup. 
Denote by $L$ the normalizer of $\Gamma_1$ in $\Gamma_2$; so $\Gamma_1 \cap \Delta \in \Sub(L)$.
Then $\rd_{L}(\Delta \cap \Gamma_1) \supseteq \rd_{\Gamma_2}(\Delta) \cap L$. In particular,
if $\Delta < \Gamma_2$ is a boomerang subgroup and $N \lhd \Gamma_2$, then 
\begin{enumerate}
    \item $\Delta \cap \Gamma_1$ is a boomerang subgroup of $\Gamma_1$, and
    \item $\Delta \cap N$ is a boomerang subgroup of $\Gamma_2$.
\end{enumerate}
A similar statement holds with $\rd$ replaced by $\rd^{lV}$. Consequently similar corollaries to (1),(2) above hold for lV-boomerangs.
\end{lemma}

\begin{proof}
Let $\gamma \in L$.
Assume $\lim_{i \to \infty} \gamma^{n_i} \Delta \gamma^{-n_i} = \Delta$.
Since intersecting with $\Gamma_1$ is a continuous map $\Sub(\Gamma_2) \to \Sub(\Gamma_1)$, we get
$$\lim_{i \to \infty} (\gamma^{n_i} (\Delta \cap \Gamma_1) \gamma^{-n_i}) =
(\lim_{i \to \infty} \gamma^{n_i} \Delta \gamma^{-n_i}) \cap \Gamma_1 =
\Delta \cap \Gamma_1.$$
\end{proof}

\begin{lemma}\label{lem:rec_in_fi}
    Let $\Gamma_1 < \Gamma_2$ be countable discrete groups and $\Delta < \Gamma_1$ a boomerang. Assume that for every $\gamma \in \Gamma_2$ there exists $k>0$ with $\gamma^k \in \Gamma_1$.
    Then $\Delta < \Gamma_2$ is a boomerang.
    In particular, this is true if $[\Gamma_2 : \Gamma_1]<\infty$. A similar statement holds for lV-boomerangs.
\end{lemma}

\begin{proof}
    It is obvious that if $\gamma^k$ is a recurrent direction, then also $\gamma$ is.
\end{proof}

\begin{lemma}\label{lem:commensurator}
    Let $\Gamma < G$ be a countable discrete subgroup and assume $g \in G$ commensurates $\Gamma$. Let $\Delta < \Gamma$ be a boomerang subgroup.
    Then, $g \Delta g^{-1} \cap \Gamma < \Gamma$ is a boomerang. A similar statement holds for lV-boomerangs.
\end{lemma}

\begin{proof}
    Since $g \Delta g^{-1} < g \Gamma g^{-1}$ is a boomerang, 
    also $g \Delta g^{-1} \cap \Gamma < g \Gamma g^{-1} \cap \Gamma$ is a boomerang by Lemma \ref{lem:rec_intersection_subgroup}.
    Now we are done by Lemma \ref{lem:rec_in_fi}.
\end{proof}

\begin{lemma} \label{lem:surj_image} Let $\varphi\colon \Gamma \rightarrow \Sigma$ be a surjective homomorphism between two countable groups. Then $\varphi(\Delta) \in \BoomlV(\Sigma)$ whenever $\Delta \in \BoomlV(\Gamma)$. 
\end{lemma}
\begin{proof}
Let $S \subset \varphi(\Delta)$ be a finite subset and $\sigma \in \Sigma$. Choose liftings $\tilde{\sigma} \in \varphi^{-1}(\{\sigma\})$ and $\tilde{S} \subset \Delta$ with $\tilde{S}$ finite and $\phi(\tilde{S})=S$. Now using the lV-boomerang condition we can find $n \in \N$ such that $\tilde{S} \subset \tilde{\sigma}^{n} \Delta \tilde{\sigma}^{-n}$, and hitting this equation with $\varphi$ gives us the desired corollary.  
\end{proof}

\begin{lemma}
The set $\Boom(\Gamma) < \Sub(\Gamma)$ is a $G_{\delta}$-set.
In particular, it is measurable.
\end{lemma}
\begin{proof}
As a compact metric space $\Sub(\Gamma)$ is separable with a countable basis for the topology $\{U_1,U_2,\ldots\}$. 
Set $C_n := \Sub(\Gamma) \setminus U_n$.
For every $n \in \N$ and $\gamma \in \Gamma$ the collection 
$$W(n,\gamma) := C_n \cup \left(\bigcup_{m \in \N} \{\Delta \in U_{n} \ | \ \gamma^m \Delta \gamma^{-m} \in U_{n}\}\right)$$
is $G_{\delta}$ as the union of a closed and an open subset in a metric space. And hence $\Boom(\Gamma)$ is $G_{\delta}$ as a countable intersection of such sets: 
$$\Boom(\Gamma) = \bigcap_{\substack{\gamma \in \Gamma, \\ n \in \N}} W(n,\gamma) .$$
\end{proof}

\begin{remark}
$\BoomlV(\Gamma)$ is also a $G_{\delta}$-set in the Chabauty topology. However, it is not a $G_{\delta}$-set in the lower Vietoris topology. The proof above works similar. In fact, if we take as a basis for the lower Vietoris topology sets of the from $\{\Env(F) \ | \ F \subset \Gamma, \ \abs{F}<\infty\}$, it is easy to see that the sets $W(n,\gamma)$ in the above proof are actually open in the Chabauty topology and hence their intersection is $G_{\delta}$. 
\end{remark}

Recall that a countable group is called LERF if every finitely generated subgroup is closed in the profinite topology, i.e. it is  an intersection of finite index subgroups. It was observed in  \cite[Thm. 3.1]{GKM:isolated} that a countable group is LERF if and only if the collection of finite index subgroups is dense. From this the following corollary immediately follows:
\begin{corollary}\label{cor:lerf}
    Let $\Gamma$ be a LERF group. Then $\Boom(\Gamma) \subset \Sub(\Gamma)$ is residual. 
\end{corollary}

Boomerang subgroups can be expected to look quite complicated. The following lemma demonstrates that finitely generated subgroups can be lV-boomerangs only for rather trivial reasons. Note again the resemblance to IRSs, since a finitely generated, ergodic IRS necessarily has a normalizer of finite index.

\begin{lemma}\label{lem:fg_recurrent}
    Let $\Gamma$ be a countable discrete group and $\Delta < \Gamma$ a finitely generated subgroup. Then $\Delta$ is an lV-boomerang subgroup if and only if its normalizer $N_{\Gamma}(\Delta)$ contains a positive power of every element of $\Gamma$.
\end{lemma}
\begin{proof}
    Let $S$ be a finite generating set for $\Delta$.
    Let $\gamma \in \Gamma$.
    By the lV-boomerang condition there exists $n \geq 1$ such that $\gamma^n S \gamma^{-n} \subset \Delta$, i.e. $\gamma^n \Delta \gamma^{-n} < \Delta$.
    But also $\gamma^{-1}$ is a lV-recurrent direction, so by the same argument,
    there exists $m \geq 1$ with $\gamma^{-m} \Delta \gamma^m < \Delta$.
    But then $\gamma^{nm} \Delta \gamma^{-nm} < \Delta < \gamma^{nm} \Delta \gamma^{-nm}$ and we conclude that $\gamma^{nm} \in N_\Gamma(\Delta)$.
\end{proof}

The following lemma is very useful for proving that lV-boomerang subgroups cannot fix points in various situations.
\begin{lemma}\label{lem:rec_subgroup_no_fixed_point}
    Let $\Gamma$ be a countable discrete group acting continuously on a locally compact, Hausdorff space $X$. Let $\Delta < \Gamma$ be a subgroup and let $x,y \in X$ be such that $\Delta x = x$ and $\Delta y \neq y$.
    Assume there exists $\gamma \in \Gamma$ such that $\lim_{n \to \infty} \gamma^n x = y$.
    Then $\gamma$ is not an lV-recurrent direction for $\Delta$.
\end{lemma}

\begin{proof}
     Take $\delta \in \Delta$ with $\delta y \neq y$.
     Let $O$ be an open neighbourhood of $y$ such that $O \cap \delta O = \emptyset$.
     Note that $\gamma^n \Delta \gamma^{-n}$ fixes $\gamma^n x$ for every $n \geq 1$.
     If $\gamma^n x \in O$ clearly $\delta \gamma^n x \neq \gamma^n x$ and thus
      $\gamma^n \Delta \gamma^{-n} \notin \Env(\delta)$.
     Therefore, $\gamma \notin \rd_{\Gamma}^{lV}(\Delta)$.
\end{proof}

\begin{remark}\label{rem:obit converging to non-fixed point}
The proof more precisely shows the following. Suppose we have
an open set $O \subset X$
and an element $\delta \in \Delta$ with $\delta O \cap O = \emptyset$,
an element $\gamma \in \Gamma$ and a sequence $n_i \to \infty$,
and an $x \in X$ such that $\gamma^{n_i}(\Delta x) \subset O$ for all $i$,
then $\gamma^{n_i} \Delta \gamma^{-n_i}$ does not converge to $\Delta$ in the lower Vietoris topology.
\end{remark}

\subsubsection*{Convergence groups}
We now use this lemma to prove Theorem \ref{thm:convergence}. Before that, let us recall a few facts about convergence groups. These were first introduced by Furstenberg \cite{Furstenberg:conv}, under the name of Dynkin groups. Gehring and Martin \cite{GM:conv}, who independently introduced this concept as part of their study of Kleinian groups via their action on the boundary of hyperbolic space, called them ``convergence groups''. The family of non-elementary convergence groups is wide and encompasses Gromov hyperbolic groups, relatively hyperbolic groups, mapping class groups of hyperbolic surfaces and the groups $\Aut(F_n)$ and $\Out(F_n)$. 

An action $\Gamma \curvearrowright X$ on a compact Hausdorff space is called a convergence action if for every infinite sequence of distinct elements $(\gamma _{n})\in \Gamma$ there exists a subsequence $(\gamma _{n_{k}})$ and points $a,r \in X$ such that $(\gamma _{n_{k}}|_{X \setminus \{r\}})$ converges uniformly on compact subsets to $a$. Such an action is called elementary if either $\Gamma$ is finite, or if the action admits orbits of size $\le 2$. Finally a {\it{non-elementary convergence group}} is any group admitting such an action. 

A direct consequence of the definition is the classification of the elements of a non-elementary convergence group into the three regular categories. Elements of finite order are referred to as {\it{elliptic}}. With every element of infinite order $g \in \Gamma$ the definition associates an attracting point $a_g$ and a repelling point $r_g$; the element is called {\it{parabolic}} if $a_g=r_g$ and {\it{hyperbolic otherwise}}. 

A non-elementary convergence action need not be minimal but it always contains a unique closed minimal $\Gamma$-invariant subset known as the {\it{limit set}} $\Lc \subset X$. 
Choose a base point $x \in X$, then
this is the set of all possible limit points of the form $\Lc = \{y \mid y= \lim_{n \rightarrow \infty} \gamma_n x\}$ with $(\gamma_n)$ ranging over all sequences in $\Gamma$ that are not eventually constant. The limit set is independent of the choice of basepoint $x$. The action $\Gamma \curvearrowright \Lc X$ is a non-elementary, minimal convergence action.

\begin{proof}[Proof of Theorem \ref{thm:convergence}]
Without loss of generality we may assume that $X = \Lc \Gamma$.
% Since we assumed that the action  $\Gamma \curvearrowright X$ is nonelementary this set is infinite. Recall that this action is minimal.
% 
Assume by contradiction that $\Lc \Delta \neq X$,
i.e. we can choose $x \in X$ such that $\Delta x$ is not dense in $X$.
Assume also by contradiction that the action of $\Delta$ on $X$ is non-trivial, i.e. there is $\emptyset \neq O_1 \subset X$ open and $\delta \in \Delta$ with $\delta(O_1) \cap O_1 = \emptyset$.
Let $O_2 \subset X \setminus \overline{\Delta x}$ be non-empty and open.
Choose a hyperbolic element $\gamma \in \Gamma$ with attracting point in $O_1$ and repelling point in $O_2$ (see for example  \cite[Lemma 2.1]{Gel:convergence}).
Then $\gamma^n(\Delta x) \subset O_1$ for all large enough $n$, and we get a contradiction by Remark \ref{rem:obit converging to non-fixed point}.
This proves that either $\Lc \Delta = X$ or $\Delta$ acts trivially on $X$ and hence must be finite.
\end{proof}

\section{The case of $\SL_n(\Z)$ with $n \ge 3$.} 
\label{sec:SLn}
This section is dedicated to an elementary proof of our main theorem in the case where $\Gamma$ is a finite index subgroup of $\SL_n(\Z)$. The Margulis normal subgroup theorem as well as the Nevo--Stuck--Zimmer theorem in this case follow as easy corollaries. The proof is inspired by \cite{Venkataramana:87} and \cite{Meiri:gen_pairs}. 

Note first that by Lemma \ref{lem:rec_in_fi} it is enough to prove the theorem when $\Gamma = \SL_n(\Z)$ for some $n \ge 3$. 
Let $B$ be the subgroup of upper triangular matrices, $U < B$ the subgroup of upper triangular matrices with $1$ on the diagonal, $D$ the diagonal matrices and $P$ the set of permutation matrices such that, if necessary, the entry in the second column is $-1$ to ensure that the determinant is $1$. If $\sigma \in \Sym(n)$ then $p_\sigma \in P$ is the corresponding permutation matrix. For $i \ne j$, the elementary matrix with $1$ on the diagonals and $k$ in position $(i,j)$ is denoted $e_{i,j}^k$.

An important ingredient of the proof is the Bruhat decomposition\footnote{For $\SL_n$ this follows directly from Gaussian elimination, or more precisely from the elementary fact that given any $A \in \SL_n(\Q)$ we can always bring it to a monomial matrix (i.e. an element of $DP$) by a applying a sequence of row and column operations of the form $A \mapsto e_{i,j}^k A$ or $A \mapsto A e_{i,j}^k$, where $1 \le i < j \le n$ and $k \in \Q$. For a generic $A$ one never encounters a zero on the anti-diagonal - corresponding to the so called {\it{open Bruhat cell}} $B p_{\omega}B$, which is Zariski open, where $\omega \in \Sym(n)$ is the permutation $\omega(i)=n+1-i$.} 
$$\SL_n(\R) = \bigsqcup_{\sigma \in \Sym(n)} B p_\sigma B = U D P U;$$ which  also works over $\Q$ or any other field. 
We repeatedly use, for $(i,j) \neq (j',i')$ and $k,k' \in \R$, the commutator formula
\[
[e_{i,j}^k,e_{i',j'}^{k'}] =
\begin{cases}
e_{i,j'}^{kk'} & j=i' \\ %\yair{and $i\ne j'$}
e_{i',j}^{-kk'} & i=j' \\ %\yair{and $j\ne i'$}
e & \text{else.}
\end{cases}
\]

\begin{lemma}\label{lem:int cell not fixing 1}
    Let $\Delta < \Gamma$ be an lV-boomerang subgroup that is not central.
    Then there exists $\sigma \in \Sym(n)$ with $\sigma(1) \neq 1$
    and $\Delta \cap B p_\sigma B \neq \emptyset$.
\end{lemma}

\begin{proof}
Consider the action of $\SL_n(\R)$ on the projective space $\PP^{1}(\R^n)$. Denote by $x_i \in \R^n$ the $i^{\text{th}}$ standard unit vector, for $1 \leq i \leq n$. 
Note that $x_1$ is a common eigenvector of the set $\bigcup_{\sigma(1)=1} B p_{\sigma} B$, so we have to prove that it is not a common eigenvector for $\Delta$.
If every $x_i$ is a common eigenvector of $\Delta$, then $\Delta$ is contained in the finite group $Q = \left\{\diag(d_1,\ldots,d_n) \ | \ d_i \in \{\pm 1\}\right\}$. Since $\Delta$ is not central, this would imply the existence of an element $\delta = \diag(d_1,\ldots ,d_n) \in \Delta$ with $d_i \ne d_j$ for some $i\ne j$. But then $e^k_{i,j} \delta e^{-k}_{i,j} \not \in Q, \forall k \ne 0$, contradicting the lV-boomerang property.
So there exists $x_j$ that is not a common eigenvector of $\Delta$. If $j = 1$, we are done. Otherwise,
if $j \ne 1$, the fact that $\lim_{k \to \infty} e_{1,j}^k [x_1] = \lim_{k \to \infty} [x_1 + k x_j] = [x_j]$ combined with Lemma \ref{lem:rec_subgroup_no_fixed_point} shows that $\Delta [x_1] \neq [x_1]$.
\end{proof}

\begin{lemma}\label{lem:double commuator matrices}
    Let $\delta \in \SL_n(\Q)$ and let $\delta= v d p_\sigma u $ be its Bruhat decomposition.
    %, where $p$ represents $w \in W$.
    Let $(d_1,d_2,\dots,d_n)$ be the diagonal elements of $d$ in that order.
    Let $a \in U$ and $k \in \R$.
    Then 
    $$\left[\left[\delta,e_{1,n}^k\right],v a v^{-1}\right] 
     = v\left[ e_{\sigma(1),\sigma(n)}^{\nicefrac{k d_{\sigma(1)}}{d_{\sigma(n)}}} , a \right]v^{-1}.$$
\end{lemma}

\begin{proof}
Recall that $e_{1,n}^k$ is central in $U$.
Then the simple calculation
     \begin{align*}
     [[\delta,e_{1,n}^k],v a v^{-1}] & = [vdp_\sigma u e_{1,n}^k u^{-1} p_\sigma^{-1} d^{-1}v^{-1} e_{1,n}^{-k},v a v^{-1}] \\
    % = [v d p_\sigma e_{1,n}^k p_\sigma^{-1} d^{-1}v^{-1} e_{1,n}^{-k},a]
    &= [v e_{\sigma(1),\sigma(n)}^{ kd_{\sigma(1)}/d_{\sigma(n)}} v^{-1} e_{1,n}^{-k} , v a v^{-1}] \\ 
     & = v e_{\sigma(1),\sigma(n)}^{kd_{\sigma(1)}/d_{\sigma(n)}} v^{-1} e_{1,n}^{-k} v a v^{-1} e_{1,n}^{k} v e_{\sigma(1),\sigma(n)}^{-kd_{ \sigma(1)}/d_{\sigma(n)}} v^{-1} v a^{-1}  v^{-1}   \\
     &= v e_{\sigma(1),\sigma(n)}^{ kd_{\sigma(1)}/d_{\sigma(n)}} a e_{\sigma(1),\sigma(n)}^{-kd_{\sigma(1)}/d_{\sigma(n)}}  a^{-1} v^{-1} \\
     &= v[ e_{\sigma(1),\sigma(n)}^{kd_{\sigma(1)}/d_{\sigma(n)}} , a]v^{-1}
     \end{align*}
proves the lemma.
\end{proof}

 \begin{lemma}\label{lem:commensurate finite index}
    Let $\Delta < \Gamma < G$ be groups such that $[\Gamma : g \Delta g^{-1} \cap \Gamma] < \infty$ for some $g \in \Comm_G(\Gamma)$. Then, $[\Gamma : \Delta] < \infty$.
\end{lemma}
\begin{proof}
We have
\begin{align*}
[\Gamma: \Delta] &= [g \Gamma g^{-1}:g \Delta g^{-1}] < [g \Gamma g^{-1}:\Gamma\cap g\Delta g^{-1}] \\
&= [\Gamma \cap g \Gamma g^{-1}:\Gamma \cap g\Delta g^{-1}] \cdot 
[g \Gamma g^{-1} : \Gamma \cap g \Gamma g^{-1}].    
\end{align*}
In the last product, the first factor is finite by assumption, and the second factor is finite because $g$ commensurates $\Gamma$.
\end{proof}

\noindent Consult \cite[Prop. 3.18]{Benoist2021} for an elementary proof of the following result by Tits and Vaserstein.

\begin{proposition}[\cite{Tits1976}] \label{thm:Tits with matrices} 
 Let $\Gamma' < \Gamma$ be a subgroup such that for all $1 \leq i, j \leq n$ with $i \neq j$ there exists an integer $k \neq 0$ with $e_{i,j}^k \in \Gamma'$. Then $[\Gamma : \Gamma'] < \infty$.
\end{proposition}

\begin{proof}
[Proof of Cor. \ref{cor:sln}]
Assume that $\Delta$ is not finite and central.
By, and using the notation of, Lemma \ref{lem:int cell not fixing 1} there exists $\delta = v d p_{\sigma} u \in \Delta$ with $\sigma(1) \neq 1$.
\vspace{3mm}

\noindent \emph{Step 1:} {\it{There is an integer $r \neq 0$ and $(i,j)$ such that $e_{i,j}^r \in v^{-1} \Delta v$.}} 
There are two cases.
\begin{description}
\item [$\sigma(n) \neq 1$] Use Lemma \ref{lem:double commuator matrices} with $a = e_{1,\sigma(1)} \in U$.
By the lV-boomerang condition there are infinitely many $k,l$ such that
   \begin{align*}
       [[\delta,e_{1,n}^k],v e_{1,\sigma(1)}^l v^{-1}] 
      = v[ e_{\sigma(1),\sigma(n)}^{ k d_{\sigma(1)}/d_{\sigma(n)}} , e_{1,\sigma(1)}^l]v^{-1} = v e_{1,\sigma(n)}^{-kld_{\sigma(1)}/d_{\sigma(n)}} v^{-1} \in \Delta,
   \end{align*}
      where the last equality follows from $\sigma(n) \neq 1$.     
      Then $e_{1,\sigma(n)}^{kld_{\sigma(1)}/d_{\sigma(n)}} \in v^{-1} \Delta v $. Note that the $d_i$'s are not necessarily integers, but they are rational. Taking this element to the power of the denominator of $d_{\sigma(1)}/d_{\sigma(n)}$ finishes this step.
      
\item [$\sigma(n) = 1$] Choose any $j \neq 1,\sigma(1)$, which exists since $n\ge3$, and use Lemma \ref{lem:double commuator matrices} with $a = e_{1,j} \in U$.
By the lV-boomerang condition there are infinitely many $k,l$ such that
   \begin{align*}
       [[\delta,e_{1,n}^k],v e_{1,j}^l v^{-1}] 
      = v[ e_{\sigma(1),1}^{ k d_{\sigma(1)}/d_{1}} , e_{1,j}^l]v^{-1} = v e_{\sigma(1),j}^{kld_{\sigma(1)}/d_{1}} v^{-1} \in \Delta,
   \end{align*}
where the last equality follows from $j \neq \sigma(1)$.
Then $e_{\sigma(1),j}^{kld_{\sigma(1)}/d_{1}} \in v^{-1} \Delta v $. Again taking this element to the power of the denominator of $d_{\sigma(1)}/d_{1}$ concludes.
\end{description} 
\vspace{3mm}

\noindent \emph{Step 2:} {\it{Replacing $\Delta$ with $v^{-1} \Delta v \cap \Gamma$ we may assume that $e_{i,j}^r \in \Delta$ for some $i,j,r$.}}
By Lemma \ref{lem:commensurator}, the subgroup $v^{-1} \Delta v \cap \Gamma$ is an lV-boomerang subgroup of $\Gamma$.
 By Lemma \ref{lem:commensurate finite index}, if $v^{-1} \Delta v \cap \Gamma$ has finite index in $\Gamma$, so does $\Delta$.
\vspace{3mm}

\noindent \emph{Step 3:} {\it{For all $(i',j')$ there is $r' \neq 0$ with $e_{i'j'}^{r'} \in \Delta$.}}
Recall that for all $(i',j')$ there are infinitely many $k$ such that
$[e_{i,j}^r,e_{i',j'}^{k}] \in \Delta$.
A look at the commutator formula reveals that
$\Delta$ contains a power of every elementary matrix in the $i-$th row,
and a power of every elementary matrix in the $j-$th column.
Propagating this argument taking suitable commutators finishes the proof.
\vspace{3mm}

\noindent \emph{Conclusion:} Now we are done by Prop. \ref{thm:Tits with matrices}.
\end{proof}

\section{Zariski density} \label{sec:BD}

This section is dedicated to the proof of Theorem \ref{thm:BorelDensity}. The proof here, which is more general than our original proof, was obtained with the help of the anonymous referee. 

If $g \in G$ is an element in an algebraic group we will denote by $H_g = \overline{\langle g \rangle}^Z$ the Zariski closure of the cyclic subgroup generated by $g$. We will rely on the following lemma. 
\begin{lemma}[{\cite[Thm. 7]{CS1993}}] \label{lem:every_infinite}
Let $\Bbbk$ be a field of characteristic zero, $g \in G$ an element in an algebraic group, such that $H_g$ is Zariski connected. Then any infinite subset of $\langle g \rangle$ is Zariski dense in $H$. 
\end{lemma}

We deduce Theorem \ref{thm:BorelDensity} as a corollary from the following, more general theorem.  
\begin{theorem} \label{thm:referee}
Let $\G$ be a semisimple, connected algebraic group defined over a field $\Bbbk$ of characteristic zero and $G = \G(\Bbbk)$ the group of $\Bbbk$-points. Let $\Gamma < G$ be Zariski dense. Then the Zariski closure of any lV-boomerang subgroup $\Delta \in \Boom(\Gamma)$ is normal in $G$. 
\end{theorem}
\begin{proof}
First we claim that it would be enough to prove the theorem under the assumption that $\Gamma$ is finitely generated. By \cite[Thm. 3]{Tits:alternative} there exists a finitely generated subgroup $\Gamma_0< \Gamma$ that is still Zariski dense. In view of Lemma \ref{lem:rec_intersection_subgroup}, the finitely generated version of our theorem implies that the Zariski closure of $\Gamma' \cap \Delta$ is normal in $G$ for every finitely generated group $\Gamma_0 < \Gamma' < \Gamma$, which clearly implies what we want. 

Assume that $\Gamma$ is finitely generated and let $N = \overline{\Delta}^Z$ be the Zariski closure of $\Delta$. For every $\delta \in \Delta$, the set 
$$T(\delta) \defeq \{g \in G \ | \ g\delta g^{-1} \in N\}$$
is Zariski closed. By \cite[Proposition 4.3]{Tits:alternative} the set $$S \defeq \{\gamma \in \Gamma \ | \ H_{\gamma} {\text{ is Zariski connected}}\}$$ is Zariski dense in $G$.

By the lV-boomerang condition we know, for every non-torsion $\gamma \in \Gamma$, that $\left|\langle \gamma \rangle \cap T(\delta) \right| = \infty$. For $\gamma \in S$ Lemma \ref{lem:every_infinite} implies that
$$H_{\delta} = \overline{\left|\langle \gamma \rangle \cap T(\delta) \right|}^{Z} \subset T(\delta),$$
and in particular $\gamma$ itself is contained in $T(\delta)$. Thus $S \subset T(\delta)$ and since $S$ is Zariski dense and $T(\delta)$ Zariski closed we conclude that $T(\delta) = G$. Finally, since $\delta$ is general, we now know that $g \Delta g^{-1} < N$ for every $g \in G$. Passing to the closure, we have $g N g^{-1} < N$ for every $g \in G$, which finishes the proof. 
\end{proof}

We now turn to the proof of Theorem \ref{thm:BorelDensity}.
\begin{proof}
In the setting of the above theorem we now assume that the kernel of the projection of $\Gamma$ to each simple factor of $G$ is finite. Let $\Delta \in \Boom^{lV}(\Gamma)$ as before and $N \lhd G$ be its Zariski closure in $G$. If $\Delta$ is not finite and central then $N$ is (almost) a product of some of the simple factors of the group $G$. In particular $\Delta$ is infinite. 
Note that if $G$ admits isomorphic simple factors the decomposition of $G$ is not unique. Nevertheless we rewrite $G$ as an (almost) product $G = G_1 G_2 \ldots G_k$ in such a way that $N = G_1 \ldots G_m$ for some $m \le k$. If $m < k$ then we have $\Delta \subset G_1 G_2 \ldots G_{k-1} \cap \Gamma$, contradicting the assumption on $\Gamma$. Thus $m=k$ and $N=G$ as required. 
\end{proof}

\section{Proof of the main theorem} \label{sec:main}

In this section we prove Theorem \ref{thm:main}, as well as a nice corollary  \ref{cor:boudec}, which was suggested to us  by Adrien Le Boudec. The strategy follows that of the previous section. 

Let $\glofield$ be a number field $\Oc$ its ring of integers, $S$ a set of valuations containing all of the infinite ones and $\Oc_S = \{b \in \glofield \ | \ v(b) \ge 0, \forall v \not \in S\}$. Let $\G$ be an absolutely almost simple $\glofield$-group. Replacing if necessary $\G$ by the connected component of its simply connected cover, we may assume that $\G$ is connected and simply connected. We fix, once and for all, a faithful $\glofield$-representation of $\G$ into $\GL_n$ and identify $\G$ with its image under this representation. Set $\Gamma = \G(\Oc_S) := \G(\glofield) \cap \GL_n(\Oc_S)$. The group $\Gamma$ does depend on some of the choices made so far, such as our choice of linear representation or passing to a simply connected cover. Also Theorem~\ref{thm:main} treats more generally any subgroup commensurable to $\Gamma$. But due to Lemmas \ref{lem:rec_intersection_subgroup}, \ref{lem:rec_in_fi} and \ref{lem:commensurator} we may, and shall, restrict our attention to this specific choice of $\Gamma$ and strive to prove that every $\Delta \in \BoomlV(\Gamma)$ is either of finite index or finite and central.

Whenever we have an algebraic group $\H < \G$, we consider it with the restriction of the representation of $\G$ into $\GL_n$ and we write its $\glofield$-points in normal font $H := \H(\glofield)$.

We fix a maximal $\glofield$-split torus $\T<\G$.
Recall that by assumption $\dim(\T) \ge 2$. Let $\Phi$ denote the (relative) root system of $\G$ with respect to $\T$. Fix a lexicographic order on the character group $X(\T)$ and let $\Phi=\Phi^{+} \sqcup \Phi^{-}$ be the partition into positive and negative roots, $\Psi \subset \Phi^{+}$ the corresponding basis and $\tilde{\alpha}$ the longest root. 

Let $\Lie(\G)$ be the Lie algebra of $\G$ and write $\frakg := \Lie(\G)(\glofield)$ its $\glofield$-points.
Let, for each $\alpha \in \Phi$, $\frakg_{\alpha} < \frakg$ be the $T$-eigenspace with eigenvalue given by the weight $\lambda_{\alpha}\colon \T \rightarrow \G_m$.
Denote by $\U_\alpha$ the corresponding root groups, and $\U := \langle \U_\alpha \mid \alpha \in \Phi^+ \rangle$.
We let $$x_{\alpha}\colon \bigoplus_{r > 0} \frakg_{r \alpha} \rightarrow U_{\alpha}$$ denote the exponential map onto the associated root groups. The map $x_{\alpha}$ is always a (polynomial) isomorphism of algebraic $\glofield$-varieties, and it is a group homomorphism if $U_\alpha$ is abelian. In the case of nonreduced root systems, where $\alpha, 2\alpha \in \Phi$, the group $U_{\alpha}$ need not be abelian and $x_{\alpha}$ is not an isomorphism of groups. We will refer to such $\alpha$ as multipliable roots. Note that $U = \langle U_{\alpha} \ | \ \alpha \in \Phi^{+} \rangle$, and that $U_{\tilde \alpha} = Z(U)$ is the center of $U$. Denote by $W = N_{G}(T)/T$ the (relative) Weyl group.

We rely on the following theorem of Raghunathan, generalizing Tits' Proposition \ref{thm:Tits with matrices}. Recall that $U_\alpha := \U_\alpha(\glofield)$.
\begin{theorem}[\cite{Rag:generators}]\label{thm:Rag}
 With the above notation, assume that $\Delta < \Gamma$ satisfies 
 \begin{equation} \label{eqn:root_intersection}
[\Gamma \cap U_{\alpha}:\Delta \cap U_{\alpha}] < \infty, \forall \alpha \in \Phi.
\end{equation}
 Then $[\Gamma : \Delta] < \infty$.
\end{theorem}
Thus to prove our main theorem we start with $\Delta \in \BoomlV(\Gamma)$ which is not finite central and strive to show that it satisfies Condition \ref{eqn:root_intersection}. 

 \subsection{Some preliminary computations}
 In this subsection $G$ can be an arbitrary group. 
 
 The proof of the following lemma was suggested by the referee. It replaces our original proof that was based on taking a linear representation and considering the Zariski closure of cyclic groups.

\begin{lemma}\label{lem:centralizer_power}
    Let $G$ be a torsion-free nilpotent group, $g \in G$. Then $C_G(g)=C_G(g^k)$ for every $k \neq 0$.
\end{lemma}

\begin{proof}
Clearly $C_G(g) < C_G(g^k)$. Conversely, assume that $h \in C_G(g^k)$. By induction on the nilpotency rank of the group we know that $[g,h] \in Z(G)$. Using this, by induction on $l$, it is easy to verify that $[g^l,h]=[g,h]^l \in Z(G)$. In particular $[g,h]^k=[g^k,h]=1$. Since $G$ is torsion free it follows that $[g,h]=1$ so that $h \in C_G(g)$ as required. 
\end{proof}

\begin{corollary}\label{cor:nilpotent_intersects_center}
    Let $N$ be a torsion free nilpotent group and let $\trivgp \ne \Delta < N$ be an lV-boomerang subgroup. Then $\Delta \cap Z(N) \neq \trivgp$.
\end{corollary}

\begin{proof}
    Take $e \neq \delta \in \Delta$ and denote by $\gamma^{0}G = G, \gamma^{i}G:=[G,\gamma^{i-1}G]$ the descending central series of $G$. If $r$ is the nilpotency rank of $G$ then $\gamma^{r}G < Z(G)$. Assume $\delta \in \gamma^i G \setminus \gamma^{i+1}G$. If $\delta \in Z(G)$, we are finished. Otherwise there exists $\epsilon \in N$ such that $[\delta,\epsilon] \neq e$. By Lemma \ref{lem:centralizer_power} for every $k > 0$ also $\delta_k := [\delta,\epsilon^k] \neq e$. Since $\Delta$ is an lV-boomerang, for infinitely many values of $k$ we have $\delta_k \in \gamma^{i+1}G \cap \Delta$ and are done by induction. 
\end{proof}

The next lemma is straightforward and we leave its proof to the reader.

\begin{lemma}\label{lem:powerformula}
    Let $G$ be a group and $x,y \in G$. Suppose that $x$ and $y$ commute with $[x^{-1},y^{-1}]$.
    Then $(xy)^p = x^p y^p [x^{-1},y^{-1}]^{\frac{p(1-p)}{2}}$.
\end{lemma}

% \begin{proof}
%     The formula is obviously correct for $p=0,1$. Let now $p\geq 1$, we use induction.
%     Then $xy [y^{-1},x^{-1}] = yx$ and
%     \begin{align*}
%         (xy)^{p+1} &= (xy)^p xy = x^p y^p [x^{-1},y^{-1}]^{\frac{p(1-p)}{2}} x y \\
%         &= x^p y^p x y [x^{-1},y^{-1}]^{\frac{p(1-p)}{2}} \\
%         &= x^{p+1} y^{p+1} [y^{-1},x^{-1}]^p [x^{-1},y^{-1}]^{\frac{p(1-p)}{2}} \\
%         &= x^{p+1} y^{p+1} [x^{-1},y^{-1}]^{\frac{p(1-p)}{2} - p} \\
%         &= x^{p+1} y^{p+1} [x^{-1},y^{-1}]^{\frac{(p+1)(1-(p+1))}{2}}.
%     \end{align*}
% \end{proof}

\subsection{The commutator argument} 
\label{sec:commutator}
We are back in the setup where $G=\G(\glofield)$.
For each $w \in W = N_{G}(T)/T$ we choose a representative $p_w \in N_{G}(T) \cap \Gamma$. %A somewhat canonical choice is $p_{w_\alpha} = x_\alpha(1) x_{-\alpha}(-1) x_\alpha(1)$ whenever $w_\alpha$ is the reflection defined by the root $\alpha \in \Psi$; so we can assume $p_w \in \Gamma$
%and $p_w x_{\alpha}(k) p_w^{-1} \in \{ x_{w(\alpha)}(\pm k) \}$ for all $w \in W, \alpha \in \Phi, k \in \glofield$ (see e.g. \cite[\S 13]{VavilovPlotkin96}).
An important ingredient of the proof of our main theorem is the Bruhat decomposition $G = \bigsqcup_{w \in W} UTp_{w}U$. Thus every $g \in G$ can be written as $g = vtpu$ with $v,u \in \U(\glofield), t \in \T(\glofield)$ and $p_{\omega}$ as above\footnote{The notation $\U(\glofield),\T(\glofield)$ comes to emphasize that the Bruhat decomposition is valid over the field, and that $t,u,v$ need not be elements of $\Gamma = \G(\Oc)$. Note though that as elements of $\G(\glofield)$ they commensurate $\Gamma$.}.  We want to understand commutators of elements given in terms of this decomposition.

\begin{lemma}\label{lem:double commutator Chevalley}
     Let $g=vtpu \in G$ be as above, $a \in U$ and $x \in U_{\tilde \alpha}$. Then the element $y := tp x p^{-1} t^{-1}$ satisfies $y \in U_{w(\tilde \alpha)}$ and
    $[[g,x],vav^{-1}] = v[y, a]v^{-1}.$
\end{lemma}
\begin{proof}
That $y = tpxp^{-1}t^{-1} \in U_{\omega(\tilde{\alpha})}$ follows directly form the way Weyl elements and torus elements act on the root groups. Using in $\stackrel{1}{=}$ and $\stackrel{2}{=}$ the fact that $x$ is central in $U$, we make the following calculation:
     \begin{eqnarray*}
     [[g,x],vav^{-1}] & \stackrel{1}{=} & [vtp x p^{-1} t^{-1}v^{-1} x^{-1},vav^{-1}]  \\ 
     & = &
     [v y v^{-1} x^{-1} , vav^{-1}] \\
     &\stackrel{2}{=} & v[y , a]v^{-1}.
     \end{eqnarray*}
Which is what we wanted. 
     %where $y := tp x p^{-1} t^{-1}$. 
\end{proof}

\begin{lemma}\label{lem:intersect_max_rootgroup}
%    Let $\Gamma < \Lambda < \G(\glofield)$.
    Let $\Delta < \Gamma$ be an lV-boomerang subgroup.
    If $\Delta$ is not finite and central, then there exists $g \in G$ and $\gamma$ non-trivial with
    $\gamma \in \U_{\tilde \alpha}(\Oc) \cap g^{-1} \Delta g$.
\end{lemma}

\begin{proof}
%    By Lemma \ref{lem:rec_intersection_subgroup} the subgroup $\Delta \cap \Gamma$ is a boomerang subgroup of $\Gamma$.
    Since $\Gamma$ is Zariski dense, Thm. \ref{thm:BorelDensity} implies that $\Delta$ is Zariski dense as well.
    Thus there exists $\delta \in \Delta$ in the open cell. Explicitly $\delta = v t p u \in \Delta$, where $p$ represents the longest Weyl element $\omega \in W$.
Let $\alpha \in \Phi^+$ be such that $\alpha$ neither commutes with, nor is opposite to, $- \tilde \alpha$.
    Then, there exists a choice of basis $\Psi'$ for the root system such that $\alpha, -\tilde \alpha$ are both positive. 
    Let $\B', \U'$ be the Borel subgroup and its unipotent radical corresponding to the basis $\Psi'$.
    Take $a \in \U_{\alpha}(\Oc)$ and $x \in \U_{\tilde \alpha}(\Oc)$.

    Since $v \in G$ commensurates $\Gamma$ we may choose $l \in \N$ with $v^{-1} a^l v \in \Gamma$.
    Then, for every $m,m' \in \Z$, by applying Lemma \ref{lem:double commutator Chevalley}, we obtain
    \begin{align} \label{eqn:first_element}
        u' := v^{-1}[[\delta,x^m],v a^{lm'} v^{-1}] v = [ y^m , a^{lm'}],
    \end{align}
    with $y \in U_{-\tilde{\alpha}} = U_{\omega({\tilde{\alpha}})} < U'$.
    Since by assumption $a \in \U_{\alpha}(\Oc) < \U'(\Oc)$ it follows from the above expression that $u' \in U'$. It can be arranged that $u'$ is non-trivial because $\alpha$ does not commute with $-\tilde\alpha$. Appealing again to the fact that $v$ commensurates $\Gamma$ we have $(u')^h \in \Gamma \cap U' = \U'(\Oc)$, for a suitable $h \in \Z$.

    Next we argue that we can choose $m,m'$ such that $u' \in v^{-1}\Delta v$ by looking at the term in the middle of Equation (\ref{eqn:first_element}).
    Since $x$ and $v a^l v^{-1}$ are lV-recurrent directions for $\Gamma$, there are $m,m'$ such that $[[\delta,x^m],v a^{lm'} v^{-1}] \in \Delta$.

    In conclusion, we obtained a non-trivial element $(u')^{h} \in v^{-1}\Delta v \cap \U'(\Oc)$.    
    By Lemma \ref{lem:commensurator} we know that $v^{-1} \Delta  v \cap \Gamma \in \BoomlV(\Gamma)$. By Lemma \ref{lem:rec_intersection_subgroup} also $v^{-1} \Delta v \cap \U'(\Oc) \in \BoomlV(\U'(\Oc))$ is an lV-boomerang. Recall that $Z(\U'(\Oc))=\U_{\tilde \alpha'}(\Oc)$, where $\tilde \alpha'$ is the maximal root with respect to $\Psi'$.
    Now use Corollary~\ref{cor:nilpotent_intersects_center} to obtain a non-trivial element $\delta' \in v^{-1} \Delta v \cap Z(\U'(\Oc))$. 
    Conjugating everything by some  $p_\eta \in \Gamma$ with $\eta(\tilde \alpha) = \tilde \alpha'$ finishes the proof.
\end{proof}

\subsection{The case of Chevalley groups.} \label{subsec:Chevalley}
In this subsection we prove the main theorem in the specific case where $\G$ is $\glofield$-split and $S$ is trivial, i.e. it is exactly the set of Archimedean valuations. We will refer to the latter assumption as the {\it{arithmetic case}}, as opposed to the $S$-arithmetic case. Note that under these assumptions all root groups $U_{\alpha}$ are isomorphic to a one dimensional vector group over $\glofield$, though in general they are higher dimensional over $\Q$.
The Lie algebra of $U_\alpha$ is isomorphic to $\glofield$, so we can write $x_\alpha \colon \glofield \to U_\alpha$ and $x_\alpha$ is in particular a group isomorphism.

We will often use
\begin{lemma}[{Chevalley commutator formula, see e.g. \cite[\S 9]{VavilovPlotkin96} or \cite[\S 33.3-33.5]{Humph:LAG}}] 
\label{lem:ChevalleyCommutator}
    Let $\alpha,\beta \in \Phi$ with $\alpha \neq \pm \beta$.
    Let $k,l \in \glofield$. 
    Then there are $N_{\alpha,\beta,i,j} \in \{\pm 1, \pm 2, \pm 3\}$ such that
    $$[x_\alpha(k),x_\beta(l)] = \prod x_{i\alpha + j\beta}(N_{\alpha,\beta,i,j} k^i l^j), $$
    where $i,j$ run over all positive integers such that $i\alpha+j\beta \in \Phi$. (Recall that if $i\alpha + j\beta \in \Phi$ with $i,j>0$ then $\alpha + \beta \in \Phi.)$
\end{lemma}
\begin{proposition}\label{prop:intersect_rootgroups_lattice}
    Let $\Gamma < \Lambda < G$.
    Let $\Delta < \Gamma$ be an lV-boomerang subgroup.
    Assume that there is $\tilde k \in \Oc$ such that $x_{\tilde \alpha}(\tilde k) \in \Delta$.
%    Assume that $\Delta \cap \cap U_\alpha(\Oc)$ is non-trivial for all $\alpha \in \Phi$.
    Then, $[\U_{\alpha}(\Oc):\Delta \cap \U_\alpha(\Oc)] < \infty$ for every $\alpha \in \Phi$. Consequently  $[\Gamma:\Delta]<\infty$.
\end{proposition}

Let us call two different bases $\Psi_1,\Psi_2$ for a root system {\it{adjacent}} if the corresponding positive roots $\Phi^{+}_1, \Phi^{+}_2$ differ only by one root, i.e. $|\Phi^{+}_1 \cap \Phi^{+}_{2}| = |\Phi^{+}_1|-1$. We will use the following.
\begin{lemma}\cite[Proposition 8.2.5]{Springer} \label{lem:basis_seq}
    Let $\Phi$ be any reduced root system. Let $\Psi$ and $\hat{\Psi}$ be two different bases for $\Phi$. Then, there exists a sequence of bases $\Psi = \Psi_1,\dots,\Psi_m = \hat{\Psi}$ such that $\Psi_i,\Psi_{i+1}$ are adjacent for all $1 \le i < m$.
\end{lemma}
\begin{remark} \label{rem:basis_seq} 
Denote the set of positive roots for $\Psi_i$ by $\Phi_i^+$.
It follows from the proof of the above lemma that $\Phi_{i+1}^+ = \left(\Phi_{i}^+ \setminus \{\alpha\}\right) \cup \{-\alpha\}$ for some basis element $\alpha \in \Psi_i$. In particular $\alpha$ is never the maximal root corresponding to either of these bases. 
\end{remark}

% \begin{lemma}\label{lem:intersect_every_rootgroup}
% %    Let $\Gamma < \Lambda < \G(\glofield)$.
%     Let $\Delta < \Gamma$ be a boomerang subgroup.
%     Assume that there is $\tilde k \in \Oc$ such that $x_{\tilde \alpha}(\tilde k) \in \Delta$.
%     Then, for every root $\alpha \in \Oc$ the intersection $\Delta \cap U_\alpha(\Oc)$ is non-trivial.
% \end{lemma}

\begin{proof}[Proof of Proposition \ref{prop:intersect_rootgroups_lattice}]
Note that the last statement follows by Theorem \ref{thm:Rag}.

\emph{Step 1:} Let $\beta \in \Phi$ be any long root. Then $x_{\beta}(k) \in \Delta$ for some $k \in \Oc \setminus \{0\}$.

Being a long root, $\beta$ is the longest root for some choice of basis. By Lemma~\ref{lem:basis_seq} we can find a sequence of adjacent bases $\Psi_1, \Psi_2, \ldots, \Psi_m$, where $\Psi_1$ is chosen so that $\tilde\alpha$ is maximal with respect to $\Psi_1$ and $\beta=\beta_m$ is the longest root with respect to $\Psi_m$. Let us denote by $\beta_i$ the longest root with respect to the basis $\Psi_i$ and by $\U_i$ the positive unipotent group corresponding to that basis. 
    We rename $x_{\beta_1}(k_1) := x_{\tilde \alpha}(\tilde k)$ and proceed by induction. Assume we have found $x_{\beta_i}(k_i) \in \Delta$. Remark \ref{rem:basis_seq} says that $\beta_i$ is always $\Psi_{i+1}$-positive, so $x_{\beta_i}(k_i) \in \Delta \cap U_{i+1}$. Thus $\Delta \cap U_{i+1} \ne \trivgp$ and Corollary  \ref{cor:nilpotent_intersects_center} yields some $k_{i+1} \neq 0$ such that $x_{\beta_{i+1}}(k_{i+1})\in \Delta$. Our desired element is $x_{\beta_m}(k_m) \in \Delta$. 

\bigskip\emph{Step 2:} Let $\Phi' \subset \Phi$ be an $A_2$ sub-root system consisting of long roots. Then  $[\U_{\beta}(\Oc):\Delta \cap \U_\beta(\Oc)] < \infty$ for every $\beta \in \Phi'$.

Let $\alpha,\beta-\alpha \in \Phi'$ be a basis for $\Phi'$. Let  
%Recall that $x_\beta \colon \glofield \to U_\beta(\glofield)$ is a $\Q$-vector space isomorphism, and 
$b_1,\dots,b_d \in \Oc$ be a $\Q$-basis of $\glofield$ and $a \in \Oc \setminus \{0\}$ such that $x_\alpha(a) \in \Delta$.
Then, by the lV-boomerang condition there exist integers $l_1,\dots,l_d > 0$ such that
$[x_{\beta-\alpha}(l_i b_i),x_\alpha(a)]= x_{\beta}(l_i a b_i) \in \Delta$
for all $i=1,\dots,k$; and $x_{\beta}(l_1 a b_1),\dots,x_{\beta}(l_d a b_d)$ form a $\Q$-basis of $U_{\beta}$.

\bigskip\emph{Step 3:} Let $\Phi' \subset \Phi$ be a $B_2$ sub-root system. Then, $\Delta$ intersects every root group of $\Phi'$ non-trivially.

Let us number the roots of $\Phi'$ by $\alpha_0,\beta_0,\alpha_1,\ldots, \beta_3$, such that $\beta_i = \alpha_{i}+\alpha_{i+1}$ for $i \in \Z/4\Z$. Let $\alpha = \alpha_0$. Since the $\beta_i$'s are long roots, we already know that $x_{\beta_i}(k_i) \in \Delta$ for some choice of $k_i \in \Oc \setminus\{0\}$. 
%    \waltraud{Problem: Now we cannot say that $x_{\beta_i}(k_i k)$ is in $\Delta$ for $k \in \Oc$.
    %I think this can be fixed as follows: Take
    %$[x_{\beta_0}(-k_0),x_{\alpha_1}(-\frac{k_1l_0}{k_0}l)]$ with $l_0 \in \Z$ big enough such that $\frac{k_1l_0}{k_0} \in \Oc$. The following seems to work.
%    }
    By the lV-boomerang condition, there is an integer $l_0 \in \N$ such that 
    $$[x_{\beta_0}(-k_0),x_{\alpha_{2}}(-l_0 k_1)]^{r_0} = x_{\alpha_1}(r_0 k_0 k_1 l_0) x_{\beta_{1}}(r_0 k_0 k_1^2 l_0^2) \in \Delta$$ for all $r_0 \in \Z$,
    and similarly there is an integer $l_1 \in \N$ such that
    $$[x_{\beta_1}(-k_1),x_{\alpha_{0}}(-l_1 k_0)]^{r_1} = x_{\alpha_1}(r_1 k_0 k_1 l_1) x_{\beta_{0}}(r_1 k_0^2 k_1 l_1^2) \in \Delta$$ for all $r_1 \in \Z$.
    Choosing $r_0=l_1$ and $r_1 = l_0$, by taking the difference of the above two elements, we get 
    \[
    x_{\beta_1}(- k_0 k_1^2 l_0^2 l_1)
    x_{\beta_0}(k_0^2 k_1 l_0 l_1^2)  \in \Delta.
    \]
    Since $x_{\alpha_2}(k)$ is a recurrent direction for $\Delta$ for all $k \in \Oc$, it is also a recurrent direction for all $k \in \glofield$ (see proof of Lemma \ref{lem:rec_in_fi}).
    So we can see that for infinitely many $l \in \Z$ the following element is in $\Delta$:
    \begin{align*}
        [x_{\beta_0}(-k_0^2 k_1 l_0 l_1^2) x_{\beta_1}( k_0 k_1^2 l_0^2 l_1),x_{\alpha_2}(-\frac{l}{k_0})] &= 
        [x_{\beta_0}(-k_0^2 k_1 l_0 l_1^2) ,x_{\alpha_2}(-\frac{l}{k_0})] \\
        &= x_{\alpha_1}(k_0 k_1 l_0 l_1^2 l)
        x_{\beta_1}(k_1 l_0 l_1^2 l^2).
    \end{align*}
    Since $l_0,l_1,l \in \Z$, $x_{\beta_1}(k_1 l_0 l_1^2 l^2) \in \Delta$
    and consequently $x_{\alpha_1}(k_0 k_1 l_0 l_1^2 l)$ is a non-trivial element in $\Delta$, which is what we wanted.

    % By the boomerang condition, there is an integer $l \in \N$ such that 
    % $$[x_{\beta_0}(-k_0k_1),x_{\alpha_{1}}(-l)] = x_{\alpha_0}(k_0k_1l) x_{\beta_{1}}(k_0k_1l^2) \in \Delta,$$ 
    % where the explicit expression for the commutator follows from Chevalley's commutator relations (Lemma \ref{lem:ChevalleyCommutator}; note that the constants appearing in Lemma \ref{lem:ChevalleyCommutator} are all trivial here, see \cite[(9.5)]{VavilovPlotkin96} or \cite[33.4 Prop.]{Humph:LAG}).
    % Since  $x_{\beta_{1}}(k_0k_1l^2) \in \Delta$ 
    % we deduce that $x_{\alpha_0}(k_0k_1l) \in \Delta$, which is what we wanted. 

\bigskip\emph{Step 4:} Let $\Phi'$ be a sub-root system of $\Phi$ isomorphic to $B_2$. Then $\Delta$ intersects $\U_{\alpha}(\Oc)$ in a finite index subgroup, for every $\alpha \in \Phi'$.

Let $\alpha,\beta \in \Phi'$ be a basis such that $\alpha$ is the short root. Recall that $U_{\alpha + \beta}$ is a $\Q$-vector space isomorphic to $\glofield$, and let $b_1,\dots,b_d \in \Oc$ be a $\Q$-basis of $\glofield$.
By the previous step, there exists $a \in \Oc \setminus \{0\}$ such that $x_\alpha(a) \in \Delta$.
Then, there exist $l_1,\dots,l_d > 0$ such that
$[x_{\alpha + \beta}(l_i b_i),x_\alpha(a)]= x_{2\alpha + \beta}(2 l_i a b_i) \in \Delta$ for all $i=1,\dots,k$, and $x_{2\alpha + \beta}(2 l_1 a b_1),\dots,x_{2\alpha + \beta}(2 l_d a b_d)$ form a $\Q$-basis of $U_{2\alpha + \beta}$. This shows that $[\U_{\alpha'}(\Oc) \cap \U_{\alpha'}(\Oc) \cap \Delta]<\infty$ for all long roots $\alpha'$.
Take now $c_1,\dots,c_d \in \Oc$ a $\Q$-basis of $\glofield$ such that $x_{2\alpha+\beta}(a^2c_1),\dots,x_{2\alpha+\beta}(a^2 c_d)$ form a basis of $U_{2\alpha +\beta} \cap \Delta$.
Again there exist integers $l_1',\dots,l_d' > 0$ such that $[x_\beta(l_i'c_i),x_\alpha(a)] = x_{\alpha + \beta}(l_i'ac_i) x_{2\alpha+\beta}(-l_ia^2c_i) \in \Delta$ for $i=1,\dots,d$.
But since $x_{2\alpha+\beta}(-l_ia^2c_i) \in \Delta$ we get that 
$x_{\alpha + \beta}(l_1'ac_1),\dots,x_{\alpha + \beta}(l_d'ac_d) \in \Delta$, and these elements form a $\Q$-basis of $U_{\alpha+\beta}$. This finishes this step and the $B_2$ case.

\bigskip\emph{Step 5:} Let $\Phi'$ be a sub-root system of $\Phi$ isomorphic to $G_2$. Then, $\Delta$ intersects every root group of $\Phi'$ non-trivially.

We number the roots of $\Psi_{\alpha,\beta}$ by $\alpha_0,\beta_0,\alpha_1,\beta_1,\ldots, \beta_5$, with $\beta_i = \alpha_{i}+\alpha_{i+1}$ for $i \in \Z/6\Z$. 
%Without loss of generality we assume $\alpha = \alpha_1$. 
Recall that the $\beta_i$ form  a sub-root system isomorphic to $A_2$.
By Step 2 we know that $[\U_{\alpha'}(\Oc) \cap \U_{\alpha'}(\Oc) \cap \Delta]<\infty$ for all long roots $\alpha'$, and by passing to a high enough multiple we can assume that $x_{\beta_i}(k) \in \Delta$ for some choice of $k \in \Oc 
    \setminus \{0\}$ and every $i \in \Z/6\Z$. 
    Let $l' \in \Z \setminus \{0\}$ such that $x_{\beta_i}(k^3l')$ and $x_{\beta_i}(k^2l')$ are contained in $\Delta$ for all $i$.
    By Lemma \ref{lem:ChevalleyCommutator} we have for every $i \in \Z/{2\Z}$:
    \begin{equation} \label{eqn:comm}
        [x_{\beta_i}(-k),x_{\alpha_{i+3}}(-l l_i)] = x_{\alpha_{i+1}}(l' l'_i k) x_{\alpha_{i+2}}({l'}^2 {l'}_i^2 k) x_{\beta_{i+2}}(k^3l' l'_i) x_{\beta_{i+1}}({l'}^3 {l'}_i^3 k^2).
    \end{equation}
    Note that the constants appearing in Lemma \ref{lem:ChevalleyCommutator} are all trivial here (see \cite[Table 4]{VavilovPlotkin96} or \cite[33.5 Prop.]{Humph:LAG}).  
    
    By the lV-boomerang condition, for every $i \in \Z/6\Z$ there exist integers $l'_i \geq 2$ such that $[x_{\beta_i}(-k),x_{\alpha_{i+3}}(- l'l'_i)] \in \Delta$; let $l_i := l'l_i'$. 
    %Since for every $l$ 
 %   Choose $l_i \in \Z$ such that
 %   both $x_{\beta_{i+2}}(k^3l)$ and $x_{\beta_{i+1}}(k^3l^2)$ are contained in $\Delta$, 
    Then we conclude from Equation \ref{eqn:comm} that 
    $x_{\alpha_{i+1}}(l_i k) x_{\alpha_{i+2}}(l_i^2k) \in \Delta$. 
    Now Lemma \ref{lem:powerformula} implies that for every $p \in \Z$ 
    \begin{small}
    \begin{eqnarray*}
    x_{\alpha_{i+1}}(p l_i k) x_{\alpha_{i+2}}(pl_i^2 k) & = & \left(x_{\alpha_{i+1}}(l_ik) x_{\alpha_{i+2}}(l_i^2k) \right)^p [x_{\alpha_{i+1}}(-l_ik), x_{\alpha_{i+2}}(-l_i^2k)]^{-\frac{p(p-1)}{2}} \\
    & = & 
    \left(x_{\alpha_{i+1}}(l_ik) x_{\alpha_{i+2}}(l_i^2k) \right)^p x_{\beta_{i+1}} \left(N(l_i+l_i^2)k \frac{p(p-1)}{2}\right) \in \Delta. 
    \end{eqnarray*}
    \end{small}
    In the last line we used again the Chevalley commutator formula from Lemma \ref{lem:ChevalleyCommutator}, with $N$ an appropriate choice of constant.
    
    In particular all of the following elements belong to $\Delta$:
    \begin{align*}
        &x_{\alpha_1}(l_0 l_1 l_2 l_3 l_4 l_5  k) x_{\alpha_2}(l_0^2 l_1 l_2 l_3 l_4 l_5  k), \\
        &x_{\alpha_2}(-l_0^2 l_1 l_2 l_3 l_4 l_5  k) x_{\alpha_3}(-l_0^2 l_1^2 l_2 l_3 l_4 l_5  k),  \\
        &x_{\alpha_3}(l_0^2 l_1^2 l_2 l_3 l_4 l_5  k) x_{\alpha_4}(l_0^2 l_1^2 l_2^2 l_3 l_4 l_5  k), \\
        &x_{\alpha_4}(-l_0^2 l_1^2 l_2^2 l_3 l_4 l_5  k)x_{\alpha_5}(-l_0^2 l_1^2 l_2^2 l_3^2 l_4 l_5  k), \\
        &x_{\alpha_5}(l_0^2 l_1^2 l_2^2 l_3^2 l_4 l_5  k)x_{\alpha_0}(l_0^2 l_1^2 l_2^2 l_3^2 l_4^2 l_5  k), \\ 
        &x_{\alpha_0}(-l_0^2 l_1^2 l_2^2 l_3^2 l_4^2 l_5  k)x_{\alpha_1}(-l_0^2 l_1^2 l_2^2 l_3^2 l_4^2 l_5^2  k),
    \end{align*}
    which shows, upon multiplying all of these in the above order, that
    \[
      x_{\alpha_1}((1- l_0 l_1 l_2 l_3 l_4 l_5)l_0 l_1 l_2 l_3 l_4 l_5  k) \in \Delta.
    \]

\bigskip\emph{Step 6:} Let $\Phi'$ be a sub-root system of $\Phi$ isomorphic to $G_2$. Then $\Delta$ intersects $\U_{\alpha}(\Oc)$ in a finite index subgroup, for every $\alpha \in \Phi'$.

Again from Step 2 we already know that 
$[\U_{\alpha'}(\Oc) \cap \U_{\alpha'}(\Oc) \cap \Delta]<\infty$
for all long roots $\alpha'$.
Let $\alpha,\beta$ be a basis such that $\alpha$ is a short root.
Let $b_1,\dots,b_d \in \Oc$ be a $\Q$-basis of $\glofield$.
Take $a \in \Oc \setminus \{0\}$ such that $x_\alpha(a) \in \Delta$.
Then, there exist integers $l_1,\dots,l_d > 0$ such that
$$[x_{\alpha+\beta}(l_i b_i),x_\alpha(a)] = x_{2\alpha+\beta}(2l_i a b_i)x_{3\alpha+\beta}(-3l_i a^2 b_i)x_{3\alpha+2\beta}(-3l_i^2ab_i^2) \in \Delta.$$
Note that the three factors on the right hand side commute.
In particular, for all $p \in \Z$, also
$$[x_{\alpha+\beta}(l_i b_i),x_\alpha(a)]^p = x_{2\alpha+\beta}(2pl_i a b_i)x_{3\alpha+\beta}(-3pl_i a^2 b_i)x_{3p\alpha+2\beta}(-3pl_i^2ab_i^2) \in \Delta.$$
Because $a,b_i \in \Oc$ and $[\U_{\alpha'}(\Oc) \cap \U_{\alpha'}(\Oc) \cap \Delta]<\infty$ for all long roots $\alpha'$, we can choose $p \neq 0$ such that $x_{3\alpha+\beta}(-3pl_i a^2 b_i) \in \Delta$
and $x_{3p\alpha+2\beta}(-3pl_i^2ab_i^2) \in \Delta$ for all $i=1,\dots,d$.
Then, also $x_{2\alpha+\beta}(2pl_1 a b_1),\dots,x_{2\alpha+\beta}(2pl_d a b_d) \in \Delta$ and these elements form a $\Q$-basis of $U_{2\alpha+\beta}$. This finishes the $G_2$ case and the proof.
\end{proof}
We now have a full proof of Theorem \ref{thm:main} in the split arithmetic case:
\begin{proof} 
By Lemma  \ref{lem:rec_in_fi} we can assume that $\Gamma = \G(\Oc)$. Now for $\Delta \in \Boom^{lV}(\Gamma)$ not central and finite, Lemma \ref{lem:intersect_max_rootgroup} yields $u \in G$ such that
$\Delta' := u\Delta u^{-1} \cap \Gamma$ intersects the root group of the maximal root non-trivially. By Lemma \ref{lem:commensurator} also $\Delta'$ is a boomerang subgroup of $\Gamma$. Proposition \ref{prop:intersect_rootgroups_lattice} shows that $\Delta'$ has finite index in $\Gamma$. Now we are done by Lemma \ref{lem:commensurate finite index}.
\end{proof}

\subsection{The arithmetic case}
We now turn to the case where $\G$ is not $\glofield$-split, but retain the assumption that $S$ is just the set of Archimedean valuations. As in the statement of the theorem we always assume that $\rk_{\glofield}(\G) \ge 2$.

In particular our root system need not be reduced anymore, so we let $\Theta \subset \Phi$ be the reduced system consisting of all non-multipliable roots. Also the root groups $U_{\alpha}$ need not be one dimensional over $\glofield$ and, in the case where $\alpha$ is multipliable, they could be two-step nilpotent groups. 

%The following could be used to establish that a subgroup is of finite index in such a group: 
%\begin{lemma}\label{lem:ccpctlattice}
%    Let $\Gamma < G$ be a discrete subgroup and $N \triangleleft G$ a closed, normal subgroup such that $\Gamma \cap N < N$ and $\Gamma N / N < \Gamma / N$ are cocompact lattices. Then, $G/\Gamma$ is compact.
%\end{lemma}

%\begin{proof}
%    Continuity of the projection $G \to G/N$ implies that $\Gamma N < G$ is closed if $\Gamma N / N \subset G/N$ is closed. Now we can apply the isomorphism theorems to obtain $(G/N)/(\Gamma N / N) \cong G/ (N \Gamma)$ and $N/(\Gamma \cap N) \cong (N \Gamma)/\Gamma$ as topological spaces.     Then $G = K_1 (N \Gamma) = K_1 K_2 \Gamma$ for suitable compact $K_1, K_2 \subset G$.
%\end{proof}

% \begin{lemma}\label{lem:BCH}
% \waltraud{Fill in correct assumptions.}
%     Let $N$ be a nilpotent group. Then $[\exp(X),\exp(Y)] \cdot \gamma^1 N =\exp([X,Y]) \cdot \gamma^1 N$.
% \end{lemma}

% \begin{proof}
%     This is just Campbell--Hausdorff, see e.g. \cite[Ch. 14b]{Milne17}.
% \end{proof}

\begin{lemma}\label{lem:vectoriso}
    Assume that $\alpha,\beta \in \Phi$ form a basis for a root system $\Phi' \subset \Phi$ of type $BC_2$, with $2\alpha \in \Phi$. Let $0 \ne Y \in \frakg_\beta$.
    Then, the map $$\frakg_\alpha \to U_{\alpha + \beta}/U_{2(\alpha + \beta)}, \quad X \mapsto \exp([X,Y]) \cdot U_{2(\alpha + \beta)}$$ is an isomorphism of $\glofield$-vector spaces.
\end{lemma}
\begin{proof}
   Recall that by our notation $\Lie(U_{\alpha+\beta}) = \frakg_{\alpha+\beta} \oplus \frakg_{2(\alpha+\beta)}$. Thus, the exponential map gives rise to an isomorphism $\frakg_{\alpha+\beta} \rightarrow U_{\alpha+\beta} / U_{2(\alpha+\beta)}$. We need to show that the map $\ad_Y \colon \frakg_\alpha \to \frakg_{\alpha + \beta}$ is an isomorphism. Indeed, viewing $\frakg_{\alpha} \oplus \frakg_{\alpha+\beta}$ as a module over the group $\SL_2(\Q) = \langle U_{\beta}(\Q), U_{-\beta}(\Q) \rangle$, we can view $\ad_{Y}$ as a raising operator which will be injective on the lowest weight space $\frakg_{\alpha}$. Since the two spaces are of the same dimension we get an isomorphism. 
  
   Next, let $X_1,X_2 \in \frakg_{\alpha + \beta}$.
   Recall that $[X_1,X_2] \in \frakg_{2(\alpha + \beta)}$.
   In particular, both $X_1$ and $X_2$ commute with $[X_1,X_2]$
   and the Campbell--Hausdorff series gives $\exp(X_1+X_2)=\exp(X_1)\exp(X_2)\exp(-\frac{1}{2}[X_1,X_2])$. As $\exp(-\frac{1}{2}[X_1,X_2])$ is in $U_{2(\alpha+\beta)}$ this establishes that the above is an isomorphism of abelian groups, and we define a $\glofield$-vector space structure on $U_{\alpha+\beta}/U_{2(\alpha+\beta)}$ via this identification. 
\end{proof}
We now turn to prove Theorem \ref{thm:main} in the arithmetic case. 
\begin{proof}[Proof of the main theorem, arithmetic case.]
We start with an lV-boomerang subgroup $\Delta < \Gamma = \G(\Oc)$, which we assume is not finite and central. By Theorem \ref{thm:Rag}, if we show that $[\U_{\alpha}(\Oc): \Delta \cap \U_{\alpha}(\Oc)]<\infty$ for every $\alpha \in \Phi$, it would follow that $\Delta$ is of finite index in $\Gamma$; which is the desired result. We start with the non-multipliable roots $\Theta \subset \Phi$. 

By Lemma \ref{lem:intersect_max_rootgroup} we obtain a non-trivial element $e \ne \delta \in \Delta \cap U_{\beta}$ for some long root $\beta \in \Theta$. Let us denote by $E_{\beta}$ the one dimensional subgroup of $U_{\beta}$ spanned by $\delta$. Now given any $\alpha \ne \beta \in \Theta$ and any one-dimensional $\glofield$-subgroup $\E_{\alpha} < \U_{\alpha}$, we have from \cite[Thm. 7.2]{BorelTits1965} a $\glofield$-split subgroup $\H < \G$ containing both $\E_{\alpha}$ and $\E_{\beta}$. 
Now $\Delta \cap \H(\Oc) \in \BoomlV(\H(\Oc))$ and the existence of a nontrivial unipotent element shows that this lV-boomerang subgroup cannot be finite and central. Using the split case of our theorem from the previous section we that $[\H(\Oc):\Delta \cap \H(\Oc)]< \infty$ and in particular $[\E_{\alpha}(\Oc): \Delta \cap \E_{\alpha}(\Oc)] < \infty$. Since this is true for any choice of one-dimensional $\glofield$-subgroup $\E_{\alpha} < \U_{\alpha}$ we conclude that $[\U_{\alpha}(\Oc): \U_{\alpha}(\Oc) \cap \Delta ] < \infty$. Finally, upon replacing the role of $\beta$ with any other long root, the same statement is true also for $\alpha = \beta$. This concludes the proof in the reduced case where $\Phi = \Theta$. 

Take now $\alpha$ with $\alpha,2\alpha \in \Phi$ and suppose $\beta \in \Phi$ is such that $\alpha$ and $\beta$ are the positive basis of a root system of type $BC_2$. Let $X_1,\dots,X_{\dim(\frakg_\alpha)}$ be a basis of $\frakg_{\alpha}$ such that $\exp(X_i) \in \Gamma$.
Take $Y \in \frakg_\beta$ with $\exp(Y) \in \Delta$ a non-trivial element.
    Since $\Delta$ is a boomerang subgroup, for every $X \in \frakg_{\alpha} \cap \exp^{-1}(\Gamma)$ there is $l > 0$ with $[\exp(l X),\exp(Y)] \in \Delta$. 
    The Campbell--Hausdorff formula gives
    \begin{align*}
        [\exp(l X), \exp(Y)] = &\exp(l[X,Y])\cdot\exp(\frac{l^2}{4} [X,[X,Y]])\cdot \\
        &\exp(-\frac{l^2}{24}[Y,[X,[X,Y]]]) \in \Delta.
    \end{align*}
    Since $U_{2\alpha+\beta}U_{2(\alpha + \beta)}$ is central in $[U_\alpha,U_\beta]$, raising this equation to the $r$-th power yields, for all $r \in \Z$
    \begin{align*}
        [\exp(l X), \exp(Y)]^r = &\exp(rl[X,Y])\cdot\exp(\frac{rl^2}{4} [X,[X,Y]])\cdot \\
        &\exp(-\frac{rl^2}{24}[Y,[X,[X,Y]]]) \in \Delta.
    \end{align*}
    We already know that $\exp \colon \frakg_{2\alpha+\beta} \to U_{2\alpha + \beta}$ is a vector space isomorphism defined over $\glofield$
    and that
    $\Delta \cap U_{2\alpha + \beta}$ is a cocompact lattice. Since, by assumption, $\exp(X),\exp(Y) \in \G(\Oc)$, 
    the commutator $\frac{l^2}{4}[X,[X,Y]]$ is contained in $\frakg_{2\alpha+\beta}(\glofield)$ and therefore, for suitable $r \neq 0$, we have that
    $\frac{r l^2}{4}[X,[X,Y]] \in \exp^{-1}(U_{2\alpha+\beta}\cap \Delta)$.
    This means that for all $i=1,\dots,\dim(\frakg_{\alpha})$    
    there exist $r_i,l_i \in \Z_+$ such that 
    $[\exp(l_i X_i), \exp(Y)]^{r_i} \in \Delta$
    and    
    $\frac{r_i l_i^2}{4} [X_i,[X_i,Y]] \in \exp^{-1}(\Delta)$, i.e.
    $\exp(\frac{r_i l_i^2}{4} [X_i,[X_i,Y]]) \in \Delta$.
    But now we obtain that 
    $$\exp(r_i l_i[X_i,Y])\exp(-\frac{r_i l_i^2}{24}[Y,[X_i,[X_i,Y]]]) \in \Delta.$$
    
    Since the $r_i l_i X_i$, with $i=1,\dots,\dim(\frakg_{\alpha})$, form a basis of $\frakg_{\alpha}$, by Lemma~\ref{lem:vectoriso}, the elements
    $\exp(r_il_i[ X_i,Y])U_{2(\alpha + \beta)}$ are a basis of the $\glofield$-vector space $\nicefrac{U_{\alpha+\beta}}{U_{2(\alpha + \beta)}}$ contained in $\Delta/U_{2(\alpha + \beta)}$. Now, since both $[U_{2(\alpha+\beta)} \cap \Gamma:U_{2(\alpha+\beta)} \cap \Delta]<\infty$ and $[\nicefrac{\Gamma U_{\alpha+\beta}}{ U_{2(\alpha+\beta)}}: \nicefrac{\Delta U_{\alpha+\beta}}{U_{2(\alpha+\beta)}}] < \infty$ we deduce that $[U_{\alpha+\beta} \cap \Gamma : U_{\alpha+\beta} \cap \Delta ] < \infty$ and we are therefore done by Theorem \ref{thm:Rag}.
\end{proof}

\subsection{The $S$-arithmetic case}

We turn to the $S$-arithmetic case. Thus $S$ is a finite set of valuations of $\glofield$ containing all the infinite places. We distinguish between the finite and infinite places by setting $S = S_{\fin} \sqcup S_{\infty}$ so that $\Oc_{S} = \{k \in \glofield \ | \ v(k) \ge 0, \ \forall v \not \in S\}$ and $\Oc = \{k \in \glofield \  | \ v(k) \ge 0, \ \forall v \not \in S_{\infty}\}$. Recall that, following a few reductions in the beginning of this section, we assumed that $\G$ is connected and simply connected, and that $\Gamma = \G(\Oc_S)$. Let $\Gamma_0 = \G(\Oc)$. Finally fix $\Delta \in \BoomlV(\Gamma)$ to be an lV-boomerang subgroup that is not finite and central. 

The group $\Gamma$ is an irreducible lattice in the product $G = \prod_{v \in S} \G(\glofield_v)$. Let us write $G = G_{\infty} \times G_{\fin}$ with $G_{\infty} = \prod_{v \in S_{\infty}} \G(\glofield_v)$ and $G_{\fin} = \prod_{v \in S_{\fin}} \G(\glofield_v)$. Denote also $K = \prod_{v \in S_{\fin}} K_v$ a maximal compact subgroup of $G_{\fin}$. In order to prove that $[\Gamma:\Delta] < \infty$ it would be enough to show that $\Delta$, too, is a lattice in $G$.

 By Lemma \ref{lem:intersect_max_rootgroup} we know that $\Delta$ contains a nontrivial unipotent element $\delta$. Replacing $\delta$ by a power we may assume it is in $\Delta \cap \Gamma_0$. Thus $\Delta_0 = \Gamma_0 \cap \Delta$ is a boomerang subgroup of $\Gamma_0$ which is not finite central and the arithmetic case of our theorem shows that $\Delta_0$ is a lattice in $G_{\infty}$. Since $K$ is compact $\Delta_0$ is also a lattice in $G_{\infty} \times K$. Let $\Omega < G_{\infty} \times K$ be a fundamental domain for $\Delta$ in this group. Since $G_{\infty} \times K$ is open in $G$ and $\Omega$ has finite measure there, it would be enough to show that the same set $\Omega$ also serves as a fundamental domain for $\Delta$ in $G$. In other words we need to show that 
 $$\Delta \Omega = \Delta \Delta_0 \Omega = \Delta (G_{\infty} \times K) = G.$$
 This in turn will follow immediately if we show that $\overline{\pi_{\fin} \Delta} = \overline{\pi_{\fin} \Gamma}$, where $\pi_{\fin} : G \rightarrow G_{\fin}$ is the projection. This last statement follows directly from Corollary 
 \ref{cor:rec_dense_ss} proven Section \ref{sec:products}. This concludes the proof. 
 \qed

\subsection{Boomerang simple groups.}
\label{sec:B_simple}
We conclude with the proof of Corollary \ref{cor:boom_simp}. In fact, we show the following more general statement.
\begin{theorem} \label{thm:almost_boom_simp}
   In the setting of Theorem \ref{thm:main} let $G = \G(\glofield)$ and $G^{+}$ the group generated by all the one dimensional unipotent groups. Then $G^{+} < \Delta$ for any lV-boomerang $\Delta \in \BoomlV(G)$ that is not finite and central.  
\end{theorem}
\begin{lemma} \label{lem:fi_finite_S}
Let $\glofield$ be a number field, $\Oc$ its ring of integers and $m \in \N$. Let $\Sigma < \glofield^m$ be a subgroup of the vector group $\glofield^m$ with the property that $[\Oc_S^m: \Oc_S^m \cap \Sigma] < \infty$ for every finite set of primes $S$. Then $\Sigma = \glofield^m$. 
\end{lemma}
\begin{proof}
Take $S$ finite and take an integer prime number with $p | [\Oc^m_S : \Oc^m_S \cap \Sigma]$. Let $T \supset S$ be such that $p$ is invertible in $\Oc_T$.
Then $p | [\Oc^m_T : \Oc^m_T \cap \Sigma]$, because there is an embedding of groups $\Oc^m_S/\Oc^m_S \cap \Sigma \to \Oc^m_T / \Oc^m_T \cap \Sigma$. In particular, by the classification of finite, abelian groups we know that $\Oc^m_T / \Oc^m_T \cap \Sigma$ has an element $[x]$ of order $p$.
However, since $p$ is invertible in $\Oc_T$, multiplying by $p$ is an isomorphism $\Oc^m_T \to \Oc^m_T$, therefore $[\Oc^m_T : \Oc^m_T \cap \Sigma] = [\Oc^m_T : \Oc^m_T \cap p\Sigma]$. But since $\Sigma$ is a subgroup we also have $p\Sigma \subset \Sigma$. Together this gives $p\Sigma=\Sigma$, leading to a contradiction with $x \notin \Sigma$ and $px \in \Sigma$.
\end{proof}
\begin{proof}[Proof of Corollary \ref{cor:boom_simp}]
Let $\Delta < G$. Assume $\Delta < Z(G)$ is not central. Then $\Delta \in \BoomlV(G)$ if and only if $[\G(\Oc_{S}): \G(\Oc_{S}) \cap \Delta] < \infty$ for every finite set of places $S$ containing all the infinite places. Indeed if $\Delta \in \BoomlV(G)$ then there is some set of places $S'$ containing all the infinite places, such that $\Delta \cap \G(\Oc_{S'}) \not < Z(\G(\Oc_{S'}))$, and it follows from Theorem \ref{thm:main} that $[\G(\Oc_{S''}):\G(\Oc_{S''})\cap \Delta] < \infty$ for every $S'' \subset S$, and a posteriori for every $S$. Conversely let $h \in G$ be any element. We can find $S$ such that $h \in \G(\Oc_S)$ and by our assumption $\gamma^n \in \Delta \cap \G(\Oc_S) < \Delta$ for some $n\in \N$ proving that $h$ is a recurrent direction for $\Delta$. 

Let $\U$ be any one-dimensional unipotent subgroup of $\G$ defined over $\glofield$. By the previous paragraph we know that $[\G(\Oc_S):\G(\Oc_S)\cap \Delta]<\infty$. Intersecting with $\U$ we have $[\U(\Oc_S):\Delta \cap \U(\Oc_S)] < \infty$ for any finite set of places $S$. Now Lemma \ref{lem:fi_finite_S} implies that in fact $\Delta > \U(\glofield)$ and consequently $\Delta > G^{+}$, which is generated by all the groups of type $\U(\glofield)$. This completes the proof. 
\end{proof}
When $\G$ is assumed to be simply connected, the Kneser--Tits conjecture, which holds for number fields (see \cite{Gille:KT}), implies that $\G(\glofield) = G^+$. Hence Corollary \ref{cor:boom_simp} as stated in the introduction follows.

The following corollary was suggested to us by Adrien Le Boudec. It can be thought of as generalizing some of the results in \cite{CP:stabilizers} and \cite{PetersonThom2016} to the boomerang setting.
\begin{corollary} \label{cor:boudec}
In the setting of Theorem \ref{thm:main}, let  $\Gamma < \Lambda < \G(\Q)$ and $\Delta \in \BoomlV(\Lambda)$. Then either $\Delta$ is finite and central or $[\Gamma: \Delta \cap \Gamma] < \infty.$ 
\end{corollary}

\begin{proof}
Let now $\Gamma < \Lambda < \G(\Q)$ and assume that $\Delta \in \BoomlV(\Lambda)$ is a boomerang subgroup that is not central and finite in $\Lambda$.
Lemma \ref{lem:rec_intersection_subgroup} guarantees that $\Delta \cap \Gamma \in \Boom(\Gamma)$. Hence, if we only show that $\Delta \cap \Gamma$ is not a finite central subgroup of $\Gamma$, we are done by Theorem \ref{thm:main}.

Since $\Lambda$ contains $\Gamma$ it is Zariski dense by Borel's density theorem. It follows from Theorem \ref{thm:BorelDensity} that also $\Delta$ is Zariski dense. Hence we can fix an element $\delta = vtpu \in \Delta$ that is contained in the open Bruhat cell. 

\vspace{3mm}

\noindent \emph{Claim:} Commutators of the form $[\delta,\gamma^{k}]$ with $\delta \in \Delta, \gamma \in \Gamma, k \in \N$ are contained in $\Gamma \cap \Delta$ for infinitely many values of $k \in \N$.
Indeed 
since $\delta \in \Comm_{G}(\Gamma)$ we can fix an $r$ such that $[\delta,\gamma^{rl}]$ is contained in $\Gamma$ for every $l \in \Z$. Take $k =lr$. Now the lV-boomerang condition guarantees that $[\delta,\gamma^k]=[\delta,\gamma^{rl}] \in \Gamma \cap \Delta$ for our fixed $r$ and infinitely many values of $l$. This shows the claim.

Let now $\alpha' \in \Phi$ be such that $\alpha'$ neither commutes with, nor is opposite to, the maximal root $\tilde \alpha$.
By the claim, for infinitely many $k \in \Z$ we have
$[\delta,x_{\tilde \alpha}(k)] \in \Gamma \cap \Delta$.
Considering the double commutator
$$\epsilon = [[\delta,x_{\tilde \alpha}(k)],v x_{\alpha}(lm) v^{-1}],$$
since $v \in \Comm_G(\Gamma)$ we can fix $m \in \N$ 
such that $vx_{\alpha}(lm)v^{-1} \in \Gamma$ for every $l \in \Z$.
So again by the claim we know that $\epsilon \in \Gamma \cap \Delta$. By Lemma \ref{lem:double commutator Chevalley}, $\epsilon$ is non-trivial unipotent and therefore cannot sit in the center of $\Gamma$.
\end{proof}

\section{Dense subgroups of simple \texorpdfstring{$p$}{p}-adic Lie groups} \label{sec:products}
This section is devoted to the proof of Theorem \ref{thm:rec_dense}. In fact as promised in the introduction we will prove the following more general version of the theorem:

\begin{theorem}\label{thm:rec_dense_ss_p}
    Let $p$ be any prime.
    Let $\G$ be a semisimple, connected algebraic group over $\Q_p$ without anisotropic factors.  Let $G$ be the subgroup of $\G(\Q_p)$ generated by the unipotent elements. Let $\Gamma < G$ be dense in the Hausdorff topology and $\Delta < \Gamma$ an lV-boomerang subgroup.
    Then $\overline{\Delta} \lhd G$ is a normal subgroup.
\end{theorem}

As a corollary we obtain:

\begin{corollary}\label{cor:rec_dense_ss}
    Let $S$ be a finite set of indices. For each $i \in S$ let  $\locfield_i$ be a non-Archimedean local field of characteristic zero.
Let $\G_i(\locfield_i)$ be the $\locfield_i$-points of a $\locfield_i$-isotropic, simple, algebraic group defined over $\locfield_i$,  $G_i = \G_i(\locfield_i)^{+}$ the subgroup generated by the unipotent elements and $G = \prod_{i \in S} G_i$. 

Let $\Gamma < G$ be a countable dense subgroup of $G$. 
\begin{enumerate}
    \item The closure $\overline{\Delta} \lhd G$ is normal in $G$ for every $\Delta \in \BoomlV(\Gamma)$.
    \item If $\left| \Gamma \cap \ker(\Pr_i) \right| < \infty$ for every $i \in S$, where $\Pr_i\colon G \rightarrow G_i$ is the projection, then $\overline{\Delta} = G$ for
    every infinite lV-boomerang subgroup of $\Gamma$. 
\end{enumerate}
\end{corollary}

\subsection{Boomerangs in $p$-adic Lie groups}
We collect in this section a few general results about lV-boomerang subgroups of dense subgroups in $p$-adic Lie groups.

\begin{lemma}\label{lem:power lie algebra}
    Let $G$ be a Lie group over a non-Archimedean local field of characteristic $0$.
    Let $H < G$ be a closed subgroup.
    Let $R := \{g \in G \mid \exists n \in \N: \, g^n \in H\}$.
    If $\overline{R}$ contains a neighborhood of $e$, then $H$ is open.
\end{lemma}

\begin{proof}
    Let $\Omega < \Lie(G)$ be a compact, open subgroup such that $\exp \colon \Omega \to G$ is defined, and such that $\exp(\Omega \cap \Lie(H))=\exp(\Omega) \cap H$; see \cite[Ch. III, \S 7, no. 2., Prop. 3]{Bourbaki1972}.
    We have to show that $\Lie(G)=\Lie(H)$.
    Let $\zeta \in \Omega$ with $\exp(\zeta) \in R$.
    There is $n>0$ with $\exp(n \zeta) = \exp(\zeta)^n \in H$, which implies $n \zeta \in \Lie(H)$ and then $\zeta \in \Lie(H)$. We showed that $\Lie(H)$ contains $\Omega \cap \exp^{-1}(R)$, whose closure has non-empty interior in $\Lie(G)$. The result follows because $\Lie(H)$ is a closed subvectorspace.
\end{proof}

\begin{proposition}\label{prop:Lie ideal}
    Let $G$ be a Lie group over a non-Archimedean local field of characteristic $0$.
    Let $\Gamma < G$ be a dense, countable subgroup and $\Delta  < \Gamma$ an lV-boomerang subgroup.
    Then $\Lie(\overline{\Delta})$ is an ideal in $\Lie(G)$.
\end{proposition}

\begin{proof}
Let $H = \overline{\Delta} < G$ be the closure of $\Delta$ inside $G$.
Let $L < \Lie(G)$ be its Lie algebra.
By dimension considerations we can find a finitely generated subgroup $\Delta' < \Delta$ whose closure $H' = \overline{\Delta'}$ shares the same Lie algebra; indeed,
$\Delta'$ can be taken to be generated by the exponentials of a basis of $L$ inside $\Omega$, where $\Omega$ is as in Lemma \ref{lem:power lie algebra}.

Applying the lV-boomerang condition we obtain for every $\gamma \in \Gamma$ an $m \in \N$ with the property that $\gamma^{m} \Delta' \gamma^{-m} < \Delta$. Passing to the closure in $G$ we see that $\gamma^{m} H' \gamma^{-m} < H$ so that $\Ad_{\gamma^m}L = L$ or, in other words, $\gamma^m \in N_G(L)$.
Since $\Gamma$ is dense in $G$,
Lemma \ref{lem:power lie algebra} implies that $N_G(L)$ is open.
Now the result follows from \cite[Ch. III, \S 9, no. 4, Prop. 10]{Bourbaki1972}.
\end{proof}

Though not relevant to the present paper, we remark that this result implies that $\overline{\Delta}$ represents an element in the structure lattice of $G$ introduced by Caprace--Reid--Willis \cite{CapraceReidWillis2017}.

\begin{lemma}\label{lem:comm into small}
    Let $G$ be a topological group and $g \in G$. Then, for every neighbourhood $U$ of the identity there exists a neighborhood $V \subset U$ of the identity with  $[g,h] \in U$ for all $h \in V$.
\end{lemma}

\begin{proof}
    The map $c \colon G \to G, g \mapsto [g,h]$ is continuous.
    Set $V := c^{-1}(U) \cap U$, it is open and contains the identity.
\end{proof}

% \begin{lemma}\label{lem:power with constant centralizer}
%     Let $G$ be a linear group over any field and $g \in G$.
%     Then, there exists $k>0$ such that $C_G(g^k)=C_G(g^{kn})$ for all $n \neq 0$.
% \end{lemma}

% \begin{proof}
% By assumption $G < \GL_m(K)$ for some $m>0$ and some field $K$.
% Since $\GL_m(K)$, with the Zariski topology, is Noetherian,
% the sequence $\overline{\langle g \rangle}^Z > \overline{\langle g^2 \rangle}^Z > \overline{\langle g^{3! } \rangle}^Z > \dots$ has to stabilize.
% This means there exists $k>0$ such that for all $n \neq 0$ we have $\overline{\langle g^k \rangle}^Z = \overline{\langle g^{kn} \rangle}^Z$.
% But then $C_{\GL_m(K)}(g^k)=C_{\GL_m(K)}(\overline{\langle g^k \rangle}^Z) = C_{\GL_m(K)}( \overline{\langle g^{kn} \rangle}^Z) = C_{\GL_m(K)}(g^{kn})$ for all $k \neq  0$.
% \end{proof}

Recall that the quasi-center of a topological group $G$, denoted $\operatorname{QZ}(G)$, is the collection of all elements with an open centralizer. This is always a subgroup. For $p$-adic Lie groups this is the kernel of the adjoint representation and in particular closed. In particular in the setting of Theorems \ref{thm:rec_dense_ss_p}, \ref{cor:rec_dense_ss}, when $\G$ is a semisimple connected algebraic group over $\Q_p$, the quasi-center of $\G(\Q_p)$ is equal to the center and finite.
\begin{proposition}\label{prop:boom in qz}
    Let $G$ be a Lie group over a non-Archimedean local field of characteristic $0$. Let $\Gamma < G$ be a dense, countable subgroup. Then every discrete  subgroup $\Delta \in \BoomlV(\Gamma)$ is contained in the quasi-center of $G$.
\end{proposition}

\begin{proof}
Let $\delta \in \Delta$. Choose $U < G$ compact and open such that $U \cap \Delta = \trivgp$. 
%Recall that $G$, being a Lie group, is locally linear, so using Ado's theorem, after shrinking $U$ if necessary, we can assume that $U < \GL_m(K)$ for some local field $K$. \waltraud{Yair, please check.}
By Lemma \ref{lem:comm into small}, there is a compact, open subgroup $V < G$ such that $[\delta,g] \in U$ for every $g \in V$.
The lV-boomerang condition implies that, for every $\gamma \in V \cap \Gamma$ there is $m>0$ with $[\delta,g^m]=e$, or in other words, $C_G(\delta)$ contains a power of every element in $\Gamma \cap V$.
Now Lemma \ref{lem:power lie algebra} delivers the desired conclusion.
%
% and such that if $g = \exp(\zeta) \in V$ and $g^n \in C_G(\delta)$ then $k\zeta \subset \Lie(C_G(\delta))$. Since $\Gamma \cap V$ is dense in $V$, by the boomerang condition, every $\gamma \in \Gamma$ has a power $\gamma^m$ with $[\delta,\gamma^m] \in \Delta \cap U =\trivgp$. In other words, the centralizer of $\delta$ contains a power of every element in $\Gamma \cap V$.
% By Lemma \ref{lem:exponential_p_adic}, the Lie algebra of the centralizer of $\delta$ is the Lie algebra of $G$, and we get that $C_G(\delta)$ is open.
\end{proof}

To tackle boomerang subgroups with compact closure, we employ Willis' theory of the scale function and tidy subgroups.
The interested reader can find an introduction into this subject in \cite{Willis2018}.

\begin{lemma}\label{lem:cpct boom tidy subgroup}
    Let $G$ be a totally disconnected, locally compact group. Let $\Gamma < G$ be a countable subgroup and $\Delta < \Gamma$ such that $\overline{\Delta}$ is compact. Let $\gamma \in \Gamma$ and let $U$ be a tidy subgroup for $\gamma$. If $\gamma$ is an lV-recurrent direction for $\Delta$, then $\Delta \cap U < \bigcap_{n \geq 0} \gamma^n U \gamma^{-n}$. 
\end{lemma}

\begin{proof}
    This is a direct consequence of \cite[Lemma 9]{Willis1994}.
\end{proof}

% \begin{proposition}
% \waltraud{This seems unrelated but interesting, shall we leave it in? Maybe rather not, because it belongs into the paper I have with Gil, Pierre-Emmanuel and Tal.}
%     Let $G$ be a topologically simple, totally disconnected, locally compact group.
%     Let $\Gamma < G$ be a dense subgroup and $\Delta < G$ a boomerang subgroup such that $\overline{\Delta}$ is compact.
%     If there exists $\gamma \in \Gamma$ and $g \in G$ such that $\{\gamma^n g \gamma^{-n} \mid n \in \Z\}$ does not have compact closure, then $\Delta$ is trivial.
% \end{proposition}

% \begin{proof}
%     By the boomerang condition, for every $\gamma \in \Gamma$, the group $\Delta$ is contained in the Levi subgroup $\operatorname{lev}(\gamma) := \{g \in G \mid \overline{\{\gamma^n g \gamma^{-n} \mid n \in \Z\}} \text{ compact}\}$ of $\gamma$, which is a closed subgroup of $G$.
%     So $\Delta < \bigcap_{\gamma \in \Gamma}\operatorname{lev}(\gamma)$, which is a closed subgroup normalized by $\Gamma$, hence by $G$.
%     By assumption it is a proper subgroup, so it must be trivial.
% \end{proof}

For the following, we refer to Wang \cite[Sect. 3]{Wang1984}.
For a $p$-adic Lie group $G$ and $g \in G$, there is a unique decomposition into $g$-invariant subspaces $\Lie(G)= P_g \oplus M_g \oplus E_g$, where $\Ad_g$ contracts $M_g$, $\Ad_{g^{-1}}$ contracts $P_g$ and $\Ad_g$ acts on $E_g$ with relatively compact orbits. Moreover, $E_g$ is a Lie subalgebra of $\Lie(G)$ that satisfies $E_g=E_{g^{-1}}$ and $E_g = \Lie(G)$ if and only if $s_G(g)=s_G(g^{-1})=1$, which is equivalent to $g$ normalizing some compact, open subgroup (see \cite{Gloeckner1998}); here $s_G \colon G \to \N_+$ is the scale function defined by Willis \cite{Willis1994}.

\begin{corollary}\label{cor:cpct in Lie}
    Let $G$ be a Lie group over a non-Archimedean local field of characteristic $0$. Let $\Gamma < G$ be a countable subgroup and $\Delta < \Gamma$ a lV-boomerang subgroup such that $\overline{\Delta}$ is compact. 
    Then, $\Lie(\overline{\Delta}) \subset \bigcap_{\gamma \in \Gamma}  E_\gamma$.

    In particular, if $G$ not uniscalar, $\Lie(G)$ is simple and $\Gamma$ is dense in $G$, then $\Delta$ is finite.
\end{corollary}

\begin{proof}
Let $\gamma \in \Gamma$ and choose $U$ tidy for $\gamma$.
Recall that $U$ is also tidy for $\gamma^{-1}$.
By Lemma \ref{lem:cpct boom tidy subgroup} we have that $\Delta \cap U < \bigcap_{n \in \Z} \gamma^n U \gamma^{-n}$.
Now \cite[Thm. 3.5]{Gloeckner1998} %(or perhaps better \cite[Thm. 3.5]{Wang1984}) 
implies that $\Lie(\bigcap_{n \in \Z} \gamma^n U \gamma^{-n}) = E_\gamma$, which implies $\Lie(\overline{\Delta}) < E_\gamma$.

For the second part, note that $\bigcap_{\gamma \in \Gamma}  E_\gamma$ is a Lie subalgebra normalized by $\Gamma$ and therefore normalized by $G$. Since $\Lie(G)$ is simple, we have that either $\bigcap_{\gamma \in \Gamma}  E_\gamma$ is trivial, or equal to $\Lie(G)$. But if $\bigcap_{\gamma \in \Gamma}  E_\gamma=\Lie(G)$, then $s_G(\Gamma)=1$, which implies that $s_G(G)=1$ by continuity of the scale function. 
%
%
    % First note that not having an eigenvalue of absolute value $1$ is an open condition, so there exists such an element $\gamma$ in $\Gamma$. Let $U$ be tidy for $\gamma$, then it is also tidy for $\gamma^{-1}$.
    % Lemma \ref{lem:cpct boom tidy subgroup} implies that $\overline{\Delta} \cap U < \bigcap_{n \in \Z} \gamma^n U \gamma^{-n}$.
    % However, using \cite[Thm. 3.5]{Gloeckner1998} we see that $\bigcap_{n \in \Z} \gamma^n U \gamma^{-n}$ is finite. This shows that $\overline{\Delta}$ is discrete and therefore $\Delta$ is finite.
    % Moreover, $C_G(\Gamma)$ contains a power of every element of $\Gamma$.
\end{proof}

If $G$ is a compactly generated, uniscalar $p$-adic Lie group, then $G$ is pro-discrete by the main result of \cite{GloecknerWillis2001}.
In particular, there are arbitrarily small compact, open, normal subgroups $U \lhd G$, and $\Gamma \cap U$ is a boomerang subgroup in $\Gamma$.
Also, if $\operatorname{QZ}(G)=G$, then $G$ is trivially uniscalar. Recently 
Caprace--Minasyan--Osin \cite{CapraceMinasyanOsin2023} construct a multitude of topologically simple $p$-adic Lie groups $G$ with abelian Lie algebra (and in particular $G=\operatorname{QZ}(G)$). Their examples are \emph{extraordinary} in the sense of Gl\"ockner \cite{Gloeckner2017}.
It would be very interesting to know something about $\overline{\Delta}$, where $\Delta < \Gamma$ is a boomerang in a dense, countable subgroup of an extraordinary $p$-adic Lie group.

\subsection{Boomerangs in dense subgroups of semisimple Lie groups}
The following lemma is probably well-known.

\begin{lemma}\label{lem:open in ss}
    Let $G = \prod_{i \in I} G_i$ be a product of locally compact groups. Let $i \in I$, and denote by $\pr_i \colon G \to G_i$ the projection.
    \begin{enumerate}
    \item
    Assume that every proper open subgroup of $G_i$ is compact and not normal.
    Let $O < G$. If $\pr_i(O)$ is non-compact and $O \cap G_i$ is open, then $G_i < O$.
    \item Assume that every proper normal subgroup of $G_i$ is contained in the center and
    $G_i$ is not 2-step-nilpotent.
    Let $N \lhd G$ be normal. If $\pr_i(N)$ is not central, then $G_i < N$. Consequently, if every factor of $G$ has the required properties, then all subnormal subgroups are normal.
    \end{enumerate}
\end{lemma}

\begin{proof}
    (1) By assumption $\pr_i(O)$ is open and non-compact in $G_i$, so $\pr_i(O) = G_i$.
%    by the Howe--Moore property.
    Since $G_i \cap O \lhd O$, we also have $G_i \cap O \lhd \pr_i(O) = G_i$. But $G_i \cap O$ is open in $G_i$, and therefore $G_i \cap O = G_i$ and $G_i < O$.

    (2) Since $\pr_i(N)$ is not central it is, by assumption, equal to $G_i$. Let $g_i = \pr_i(g)$ and $h_i \in G_i$ with $[g_i,h_i] \not \in ZG_i$. Then $[g,h_i]=[g_i,h_i] \in N \cap G_i$, and since $N \cap G_i$ is normal in $G_i$, it follows that $G_i < N$.
\end{proof}

\begin{proof}[Proof of Theorem \ref{thm:rec_dense_ss_p}]
    We write $\G = \G_1 \dots \G_n$, where the $\G_i$ are the simple factors of $\G$, and $G_i$ the subgroup of $\G(\Q_p)$ generated by the unipotent elements.
    Set $H := \overline{\Delta}$.
    By Proposition \ref{prop:Lie ideal} the Lie algebra $\Lie(H)$ is an ideal in $\Lie(G)$, hence it is the direct sum $\bigoplus_{i \in I} \Lie(G_i)$ with $I$ a subset of the indices, and $H \cap G_i$ is open in $G_i$ for $i \in I$. We have to show that $G_i < N_G(H)$ for every $i=1,\dots,n$; and we will show something stronger. Namely, $G_i < H$ if $i \in I$ and $G_i < C_G(H)$ otherwise.

    By Lemma \ref{lem:surj_image}, if $\Delta < \Gamma$ is an lV-boomerang, then $\pr_i(\Delta) < \pr_i(\Gamma)$ is an lV-boomerang.

    For $i \in I$, $\pr_i(\Delta)$ is infinite so by Corollary \ref{cor:cpct in Lie} it cannot be precompact. Hence the closure of $\pr_i(H)$ is not compact, and since $H \cap G_i < G_i$ is open, we get from Lemma \ref{lem:open in ss}(1) that $G_i < H$.

    Consider now $G' := \prod_{i \notin I} G_i$ and let $\pr \colon G \to G'$ be the projection.    
 By the previous paragraph we know that $\prod_{i \in I} G_i < H$, so $\pr(H) = H \cap G'$ and in particular $\pr(H) < G'$ is discrete.

    By Proposition \ref{prop:boom in qz} we get $\pr(\Delta) < \operatorname{QZ}(G')$.
    Since the quasi-center equals the center, which is closed, we obtain $\pr(H) < Z(G)$ and also $G' < C_G(H)$.

%     For $i \in I$, let $\pr_i \colon G \to G_i$ be the projection. 

%     If $\pr_i(H)$ has compact closure, then $\overline{\pr_i(\Delta)}$ is compact.
%     We get from Corollary \ref{cor:cpct in Lie} that $\overline{\pr_i(\Delta)}$ is finite, and in particular $i \notin I$.
% %    and by Proposition \ref{prop:boom in qz}, $\pr_i(\Delta) < \operatorname{QZ}(G_i)$, and the quasi-center of $G_i$ equals the center of $G_i$. Thus $G_i < C_G(\Delta) = C_G(\overline{\Delta})$.

%     If $i \in I$, then the closure of $\pr_i(H)$ is not compact, then,
%     since $H \cap G_i < G_i$ is open, we get from Lemma \ref{lem:open in ss}(1) that $G_i < H$.
% %
    % \emph{Case 1:} $\pr_i(\Delta)$ discrete.
%
    % If $i \notin I$, then $\pr_i(H)$ is discrete and thus $\pr_i(\Delta)$ is discrete. By Proposition \ref{prop:boom in qz}, $\pr_i(\Delta) < \operatorname{QZ}(G_i)$, and the quasi-center of $G_i$ equals the center of $G_i$. Thus $G_i < C_G(\Delta) = C_G(\overline{\Delta})$. 
    
    % \emph{Case 2:} $\pr_i(H)$ has compact closure.
    % %$\overline{\pr_i(\Delta)}$ compact
%
    % If $\pr_i(H)$ is compact, then $\overline{\pr_i(\Delta)}$ is compact.
    % We get from Corollary \ref{cor:cpct in Lie} that $\overline{\pr_i(\Delta)}$ is finite and we are done by Case 1.
%
%    \emph{Case 3:} $i \in I$ and
\end{proof}

\begin{proof}[Proof of Corollary \ref{cor:rec_dense_ss}]
    Let $H := \overline{\Delta}$. Let $G_i$ be one of the factors of $G$.
    We denote by $\pr_i \colon G \to G_i$ the projection.
%    We have to show that if $\pr_i(\Delta)$ is not central, then $G_i < H$.
    %, let $p$ be the characteristic of the residue field.
    
    For a prime $p$, let $G_p$ be the product of all the factors of $G$ such that the characteristic of the residue field is $p$ (i.e. $\Q_p < \locfield_i$).
    In particular, $G_p$ is a semi-simple $p$-adic Lie group and satisfies the assumptions of Theorem \ref{thm:rec_dense_ss_p}.
    %, and it is the $k$-points of an algebraic group defined over a common finite extension of the $\locfield_i$. \waltraud{Is semisimplicity preserved under passing to finite extensions?}
    We have that $G = \prod_p G_p$, and \cite[Cor. 2.3]{Gloeckner2006} implies that $\prod_p H \cap G_p \lhd H$ is an open, normal subgroup.
    We denote by $\pr_p \colon G \to G_p$ the projection.

    \emph{Step 1:} $H \cap G_p \lhd G_p$

    Since $\Delta < \Gamma$ is an lV-boomerang, also $\pr_p(\Delta) < \pr_p(\Gamma)$ is an lV-boomerang. Theorem \ref{thm:rec_dense_ss_p} gives us that $\overline{\pr_p(\Delta)} \lhd G_p$.
    We have $$H \cap G_p = \pr_p(H \cap G_p) \lhd \pr_p(H) < \overline{\pr_p(\Delta)} \lhd G_p.$$
    Taking closures, since $\overline{\pr_p(H)} = \overline{\pr_p(\Delta)}$,
    we see that $H \cap G_p$ is subnormal in $G_p$.
    Lemma \ref{lem:open in ss}(2) gives us that $H \cap G_p \lhd G_p$.
    
    \emph{Step 2:} If $\pr_i(H \cap G_p)$ is discrete, then $\pr_i(H)$ is central.

    Assume $\pr_i(H \cap G_p)$ is discrete and let $\delta \in \Delta$.
    To show that $\pr_i(\delta)$ is central, it suffices to show that $\pr_i(\delta)$ is in the quasi-center of $G_i$.
    Let $O \subset G$ be an open subset such that $H \cap O = \prod_p H \cap G_p$.
    Let $U < G$ be a compact, open subgroup contained in $O$, and such that $\pr_i(U) \cap \pr_i(H \cap G_p) = \trivgp$.
    Let $V < U$ be a compact, open subgroup such that $[\delta,g] \in U$ for every $g \in V$.
    Let $\gamma \in \Gamma \cap V$.
    Then there is $m>0$ with $[\delta,\gamma^m] \in \Delta \cap U$,
    and because $U \subset O$ we see that
    $[\pr_i(\delta),\pr_i(\gamma)^m] \in \pr_i(H \cap G_p) \cap \pr_i(U),$
    which implies $\pr_i(\gamma)^m \in C_G(\pr_i(\delta))$.
    But $\pr_i(\Gamma)$ is dense in $G_i$, so from Lemma \ref{lem:power lie algebra} we see that $\pr_i(\delta)$ is in the quasi-center and therefore in the center of $G_i$.
    It follows that $\pr_i(H) < Z(G)$.
    
    % Since $\pr_p(\Delta) < \pr_p(\Gamma)$ is an inner boomerang subgroup whose projection on $G_i$ is not central, Theorem \ref{thm:rec_dense_ss_p} implies that $G_i < \overline{\pr_p(\Delta)}$.
    % Let $\delta \in \Delta$ be such that $\pr_i(\delta)$ is not central.
    % Let $O \subset G$ be an open subset such that $H \cap O = \prod_p H \cap G_p$.
    % Let $U < G$ be an arbitrarily small compact, open subgroup contained in $O$.
    % Let $V < U$ be a compact, open subgroup such that $[\delta,g] \in U$ for every $g \in V$. Choose $\gamma \in \Gamma$ such that $\pr_i(\gamma^k)$ does not commute with $\pr_i(\delta)$ for any $k \neq 0$.
    % Choose $m>0$ with $\delta' := [\delta,\gamma^m] \in \Delta$.
    % In particular $\delta' \in O$ and therefore $\pr_p(\delta') \in H \cap G_p$.
    % As $U$ was arbitrarily small and $\pr_i(\delta')$ is non-trivial, it follows that $\pr_i(H \cap G_p)$ is not discrete.

    \emph{Step 3:} Proof of Corollary \ref{cor:rec_dense_ss}(1).

    If $\pr_i(H)$ is central, then $G_i < C_G(H)$.
    If $\pr_i(H)$ is not central, then by the previous step, already $\pr_i(H \cap G_p)$ is non-discrete, and $H \cap G_p$ is infinite. Now that $H \cap G_p \lhd G_p$ proves, with Lemma \ref{lem:open in ss}(2), that $G_i < H \cap G_p < H$.

    \emph{Step 4:} Proof of Corollary \ref{cor:rec_dense_ss}(2).

    Since $\Delta$ is infinite and the kernel of the projections are finite, $\pr_i(\Delta)$ is never finite. Now (2) follows from (1).
\end{proof}

 \bibliographystyle{alpha}
 \bibliography{rec}
 
%\noindent {\sc Yair Glasner.} Department of Mathematics,
%Ben-Gurion University of the Negev.
%P.O.B. 653,
%Be'er Sheva 84105,
%Israel.
%{\tt yairgl\@@math.bgu.ac.il}

%\bigskip

%\noindent {\sc Waltraud Lederle.} Universit\'{e} Catholique de Louvain, Institut de Recherche en Math\'ematique et Physique, Chemin du Cyclotron 2, Box L7.01.02, 1348 Louvain-la-Neuve, Belgium
%{\tt waltraud.lederle\@@uclouvain.be}

%\bigskip

\end{document}